\documentclass[hidelinks,onefignum,onetabnum]{siamart250211}

\usepackage{lipsum}
\usepackage{amsfonts}
\usepackage{graphicx}
\usepackage{epstopdf}
\usepackage{algorithmic}
\ifpdf
  \DeclareGraphicsExtensions{.eps,.pdf,.png,.jpg}
\else
  \DeclareGraphicsExtensions{.eps}
\fi

\usepackage{mathrsfs}
\usepackage{amsmath}
\usepackage{indentfirst}
\usepackage{amssymb}
\usepackage{bm}
\usepackage{subfig}

\def\XXint#1#2#3{{\setbox0=\hbox{$#1{#2#3}{\int}$ }
\vcenter{\hbox{$#2#3$ }}\kern-.6\wd0}}

\def\ds{\,{\textrm{d}}s}

\newsiamremark{remark}{Remark}
\newsiamremark{hypothesis}{Hypothesis}
\crefname{hypothesis}{Hypothesis}{Hypotheses}
\newsiamthm{claim}{Claim}
\newsiamremark{fact}{Fact}
\crefname{fact}{Fact}{Facts}

\headers{HIGHER-ORDER FAR-FIELD BOUNDARY CONDITIONS FOR DISLOCATION}{X. Wei, J. Braun, Y. Wang and L. Zhang}

\title{Higher-Order Boundary Conditions for Atomistic Dislocation Simulations\thanks{Submitted to the editors DATE.  
\funding{The authors are partially supported by NSFC grant 12271360. LZ is also partially supported by the Shanghai Municipal Science and Technology Project 23JC1402300 and the Fundamental Research Funds for the Central Universities.}}}
\author{Xinyi Wei\thanks{School of Mathematical Sciences, Institute of Natural Sciences and MOE-LSC, Shanghai Jiao Tong University, Shanghai 200240, China
  (\email{xinyiwei@sjtu.edu.cn}, \email{lzhang2012@sjtu.edu.cn}).}
  \and Julian Braun\thanks{Department of Mathematics, Heriot-Watt University, Edinburgh, EH14 4AS, UK (\email{j.braun@hw.ac.uk}).}
\and Yangshuai Wang\thanks{Corresponding author. Department of Mathematics, Faculty of Science, National University of Singapore, 10 Lower Kent Ridge Road, Singapore
  (\email{yswang@nus.edu.sg}).}
\and Lei Zhang\footnotemark[2]}

\usepackage{amsopn}

\DeclareMathOperator{\divo}{div}

\ifpdf
\hypersetup{
  pdftitle={An Example Article},
  pdfauthor={D. Doe, P. T. Frank, and J. E. Smith}
}
\fi

\usepackage{comment}
\usepackage{multirow}
\usepackage{array}
\usepackage{graphicx}
\usepackage{makecell}

\definecolor{lzcol}{rgb}{1, 0, 0}

\newcommand{\xy}[1]{{\color{brown} #1}}
\definecolor{yscol}{HTML}{6622AA}

\begin{document}

\maketitle

\begin{abstract}
We present a higher-order boundary condition for atomistic simulations of dislocations that address the slow convergence of standard supercell methods. The method is based on a multipole expansion of the equilibrium displacement, combining continuum predictor solutions with discrete moment corrections. The continuum predictors are computed by solving a hierarchy of singular elliptic PDEs via a Galerkin spectral method, while moment coefficients are evaluated from force-moment identities with controlled approximation error. A key feature is the coupling between accurate continuum predictors and moment evaluations, enabling the construction of systematically improvable high-order boundary conditions. We thus design novel algorithms, and numerical results for screw and edge dislocations confirm the predicted convergence rates in geometry and energy norms, with reduced finite-size effects and moderate computational cost.
\end{abstract}



\section{Introduction}
\label{sec:intro}

Dislocations play a central role in determining the mechanical and physical properties of materials. Their presence typically distorts the host lattice, giving rise to long-range elastic fields~\cite{eshelby1953anisotropic}. This work is concerned with the precise characterization of the geometric and energetic properties of dislocations in crystalline solids~\cite{groger2008multiscale, olson2023elastic, steinhauser2017computational, wang2024theoretical, yip2007handbook}. Numerical simulations of dislocations necessarily confine the system to finite computational domains. In practice, the equilibria are often computed within standard supercells, where artificial (e.g., clamped or periodic) boundary conditions must be imposed~\cite{cai2003periodic, li2005multiscale, li2006variational}. However, the convergence of the supercell method with respect to cell size is often slow~\cite{braun2020sharp, braun2025higher, EOS2016, trinkle2008lattice}. The choice of boundary conditions significantly affects cell-size effects, motivating the development of {\it higher-order boundary conditions} to accelerate convergence.

A key step in this endeavor is to analyze the elastic far-fields induced by dislocations, often through their low-rank structure. This is typically achieved using continuum linear elasticity and the defect dipole tensor~\cite{eshelby1956continuum, nowick1963anelasticity}. Braun et al.~\cite{braun2019effect, braun2022asymptotic} developed a unified mathematical framework for modeling elastic far-fields via low-order defect structures. Within this framework, the defect equilibrium is decomposed into continuum predictors and discrete multipole terms. This decomposition leads to improved convergence of supercell approximations with respect to cell size theoretically. Recently, Braun et al.~\cite{braun2025higher} developed a numerical framework for point defects. 

Extending such approaches to dislocations presents significant challenges, including the incorporation of multipole moments from infinite lattices into finite computational domains and the accurate resolution of singularities in higher-order continuum predictor equations near the dislocation core; to address these, we extend the theoretical framework of \cite{braun2022asymptotic} by formulating high-order boundary conditions and developing a corresponding numerical scheme. We introduce rescaled variables to remove core singularities, enabling the accurate and efficient solution of the higher-order equations, and employ a tailored Galerkin spectral method that leverages spectral convergence for high accuracy at low computational cost. To improve multipole representations within finite domains, we combine a continuous version of multipole expansions with an iterative refinement strategy~\cite{braun2025higher} to yield systematically improvable boundary conditions with controlled accuracy. Although the main overhead arises in geometry optimization, the overall efficiency is substantially enhanced by our boundary formulation's accelerated convergence, and validation on screw and edge dislocations shows our method improves geometry and energy convergence rates, achieving higher accuracy with faster domain-size convergence and lower overall cost.


This paper focuses on a comprehensive understanding of higher-order boundary conditions for dislocation simulations, aiming to achieve improved accuracy at moderate computational cost. Future work will explore broader applications of the proposed framework to multiscale coupling methods~\cite{chen2022qm, wang2021posteriori, wang2024posteriori}. However, extending the method to more complex scenarios, such as cracks or grain boundary structures, poses fundamental challenges that lie beyond the scope of the present study and will require the development of new theoretical tools.



\subsection*{Outline} 

The paper is organized as follows. In \S~\ref{sec:disloc}, we present the variational formulation and multipole expansion for dislocation equilibria. \S~\ref{sec:Scheme} introduces a theoretical framework for accelerated defect simulations using a continuous multipole expansion. In \S~\ref{sec:pde}, we develop numerical methods for solving the resulting higher-order predictor equations and design the core algorithm (Algorithm~\ref{alg:moment_iter_a}). \S~\ref{sec:numexp} presents numerical results for screw and edge dislocations. \S~\ref{sec:conclusion} concludes with a summary and outlook. Technical details and proofs are provided in Appendices~\ref{app:premilinary results}--\ref{app:supply and proof}.

\subsection*{Notation}

We denote the abstract duality pairing between a Banach space and its dual by $\langle \cdot, \cdot \rangle$. 
Norms are denoted by $|\cdot|$ for Euclidean or Frobenius norms, and $\|\cdot\|$ for operator norms. The notation $|A| \lesssim B$ means $|A| \leq C B$ for a constant $C$ independent of problem parameters (e.g., lattice size, test functions). $B_r(x)$ denotes the ball centered at $x$ with radius $r$.

We introduce a $k$-tuple of vectors in $\mathbb{R}^d$, denoted as $\boldsymbol{\sigma} = (\sigma^{(1)}, \ldots, \sigma^{(k)}) \in (\mathbb{R}^d)^k$. The $k$-fold tensor product is expressed as $\boldsymbol{\sigma}^\otimes := \bigotimes_{i=1}^k \sigma^{(i)} = \sigma^{(1)} \otimes \cdots \otimes \sigma^
{(k)}$,
and the vector space spanned by these tensor products is identified as $(\mathbb{R}^d)^{\otimes k}$. We also consider the symmetric tensor product by $\boldsymbol{\sigma}^\odot := {\rm sym~} \boldsymbol{\sigma}^\otimes := \sum_{g \in S_k} g(\boldsymbol{\sigma})^\otimes /k!$, where $S_k$ is the symmetric group encompasses all permutations of $\{1, \ldots, k\}$. For a single vector $v \in \mathbb{R}^d$, we denote its $k$-fold tensor product by $v^{\otimes k} := v \otimes \cdots \otimes v \in (\mathbb{R}^d)^{\otimes
k}$. 

Consider two tensors $\mathbf{C}, \mathbf{U} \in (\mathbb{R}^\mathcal{R})^{\otimes k}$ with a finite interaction range $\mathcal{R}$. Suppose $\mathbf{C}$ is expressed in the natural basis ${E_{\boldsymbol{\rho}}}$ of $(\mathbb{R}^d)^{\otimes k}$ as $\mathbf{C} = \sum_{\boldsymbol{\rho} \in \mathcal{R}^k} \mathbb{C}_{\boldsymbol{\rho}} E_{\boldsymbol{\rho}}$, then the contraction can be written as $\mathbf{C}:\mathbf{U} = \sum_{\rho\in\mathcal{R}}\mathbf{C}_{\rho}\mathbf{U}_{\rho}$.
\section{Background: Modeling of Dislocations}
\label{sec:disloc}

This work focuses on the modeling of crystalline dislocations, with particular attention to single antiplane screw and edge dislocations. Such settings enable the development of rigorous computational  framework. Following the construction in~\cite{braun2019effect, EOS2016, hudson2015analysis}, we first present the atomistic model in \S~\ref{subsec:atom-model}. In \S~\ref{sec:sub:equilibrium and expansion}, we introduce the multipole expansion of dislocation equilibrium, which forms the theoretical cornerstone of this work.

\subsection{Atomistic model}
\label{subsec:atom-model}

We consider a straight dislocation that is periodic along the dislocation line, following the setup in~\cite{EOS2016}. This allows a reduction to a two-dimensional lattice model on the plane normal to the dislocation line~\cite{olson2023elastic}. The homogeneous crystal is represented by a Bravais lattice \( \Lambda := A \mathbb{Z}^2 \), where \( A \in \mathbb{R}^{2 \times 2} \) is a non-singular matrix. Atomic displacements are defined as \( u : \Lambda \to \mathbb{R}^N \), where \( N = 1 \) for anti-plane screw dislocations and \( N = 2 \) for edge dislocations.
For $\ell, \rho \in \Lambda$, we denote discrete differences by $D_\rho u (\ell) := u(\ell + \rho) - u(\ell)$.
We assume a finite interaction neighbourhood $\mathcal{R} \subset \Lambda-\ell$, which spans the lattice ${\rm span}_{\mathbb{Z}}\mathcal{R} = \Lambda$ throughout the discussion. Define the discrete difference stencil $Du(\ell):=D_{\mathcal{R}}u(\ell):=(D_{\rho}u(\ell))_{\rho \in \mathcal{R}}$. To introduce higher discrete differences we denote by $D_{\bm \rho} u := D_{\rho_1}\ldots D_{\rho_j} u$, for ${\bm \rho}\in(\mathcal{R})^j$.

We define two useful discrete energy spaces by
\begin{align}
\mathcal{H}^1 &= \mathcal{H}^{1}(\Lambda) = \{u: \Lambda \rightarrow \mathbb{R}^N \,\big|\, \Vert Du \Vert _{\ell^2} < \infty \},\\
\mathcal{H}^{\rm c} &= \mathcal{H}^{\rm c}(\Lambda) = \{ u: \Lambda \rightarrow \mathbb{R}^N \,\big|\, \text{supp}(Du)\text{ bounded} \}.
\end{align}
where $\mathcal{H}^{\rm c}$ is a dense subspace of $\mathcal{H}^{1}$ with compact support~\cite{luskin2013atomistic}. 

Let \( \hat{x} \in \mathbb{R}^2 \) denote the dislocation core, and define the branch cut $\Gamma := \{ x \in \mathbb{R}^2 : x_2 = \hat{x}_2,\; x_1 \geq \hat{x}_1 \}$, chosen such that \( \Gamma \cap \Lambda = \emptyset \). Due to the topological singularity, dislocations in an infinite lattice exhibit infinite elastic energy. Following~\cite{braun2022asymptotic}, we decompose the total displacement $u = u_0^{\rm C} + \bar{u}$, into a far-field predictor \( u_0^{\rm C} \) and a finite-energy core corrector \( \bar{u} \). To construct \( u_0^{\rm C} \), we begin with the continuum linear elasticity (CLE) solution \( u^{\rm lin} \in C^\infty(\mathbb{R}^2 \setminus \Gamma; \mathbb{R}^N) \), which satisfies
\begin{equation}
\label{eq:u0}
\begin{alignedat}{2}
- \operatorname{div} \, \mathbb{C}[\nabla u^{\rm lin}] &= 0,           &\qquad& \text{in } \mathbb{R}^2 \setminus \Gamma, \\
u^{\rm lin}(x+) - u^{\rm lin}(x-) &= -\mathsf{b},                       &\qquad& \text{for } x \in \Gamma \setminus \{\hat{x}\}, \\
\nabla_{e_2} u^{\rm lin}(x+) - \nabla_{e_2} u^{\rm lin}(x-) &= 0,      &\qquad& \text{for } x \in \Gamma \setminus \{\hat{x}\}, \\
|\nabla u^{\rm lin}(x)| &\to 0,                                        &\qquad& \text{as } |x| \to \infty, \\
\int_{\partial B_1(\hat{x})} \mathbb{C}[\nabla u^{\rm lin}] \nu \, \mathrm{d}\sigma &= 0. &\qquad& 
\end{alignedat}
\end{equation}
Here, \( \mathbb{C} \) denotes the linearised Cauchy--Born tensor~\cite{hudson2012stability, ming2007cauchy, ortner2013justification}. The higher-order continuum predictor will be discussed in \S~\ref{sec:sub:equilibrium and expansion}.
For the case of antiplane screw dislocations, we have $u_0^{\rm C} = u^{\rm lin}$. For edge dislocations, in order to ensure that the corrected predictor $u_0$ possesses $C^\infty$ regularity in the half-space, we define the predictor~\cite{EOS2016}:
\begin{equation}
  \label{eq:disl:defn_u0}
  u_0^{\rm C}(x) := u^{\rm lin}\big(\xi^{-1}(x)\big), \quad {\rm where} \quad \xi(x) := x
  - {\sf{b_{12}}} \frac{1}{2\pi} \eta\big(|x-\hat{x}|/\hat{r}\big) \arg(x-\hat{x}),
\end{equation}
$\arg(x)$ denotes the angle in $(0,2\pi)$ between $\mathsf{b}_{12}\propto e_1$ and
$x$, and $\eta \in C^\infty(\mathbb{R})$ with $\eta = 0$ in $(-\infty, 0]$,
$\eta = 1$ in $[1, \infty)$ and $\eta' > 0$ in $(0, 1)$. 

We assume throughout that the site energy potential $V \in C^{K}(\mathbb{R}^{N\times \mathcal{R}})$ with $K \geq 4$, and that it satisfies the point symmetry $V(A) = V((-A_{-\rho})_{\rho \in \mathcal{R}})$. Furthermore, we assume $V({0}) = 0$ and that $V$ is invariant under lattice slip. We provide further details of these assumptions and symmetry in Appendix~\ref{app:V}. 

The energy difference functional for displacement field is given by:
\begin{equation}
\label{eq:disl:Ediff}
\mathcal{E}(u) := \sum_{\ell \in \Lambda} \Big(V\big(Du(\ell) \big) -V\big(Du^{\rm C}_0(\ell)\big) \Big).
\end{equation}
The Hessian operator of $\mathcal{E}(u)$ at the reference state is defined as:
\begin{equation}
H[u,v] = \delta^{2}\mathcal{E}(0)[u,v] = \mathop{\sum}\limits_{l \in \Lambda} \nabla^{2}V(0)[Du,Dv].
\end{equation}
For lattice displacements $u : \Lambda \to \mathbb{R}^N$ that are close to equilibrium, we define
\begin{equation}
\label{eq:force}
H[u](\ell) :=-{\rm Div}\big(\nabla^2V(0)[Du]\big).
\end{equation}
with ${\rm Div} A = -\sum_{\rho\in \mathcal{R}}D_{-\rho}A_{\cdot\rho}$. 
We assume throughout that the Hamiltonian $H=\delta^2 \mathcal{E}({0})$ is stable~(cf.~\cite[Eq.~(4)]{braun2022asymptotic}).
For a stable operator $H$, there exists a {\it lattice Green's function} (inverse of $H$) $\mathcal{G}:\Lambda \rightarrow \mathbb{R}^{N\times N}$ such that 
\begin{eqnarray}\label{eq:def_lattice_G}
H[\mathcal{G}e_k](\ell) = e_k \delta_{\ell, 0}, \qquad \textrm{for}~1 \leq k \leq N. 
\end{eqnarray}
We write $\mathcal{G}_k:=\mathcal{G}e_k$ for simplicity.



\subsection{Multipole expansion of dislocations equilibrium}
\label{sec:sub:equilibrium and expansion}
In this section, we review the equilibrium formulation and multipole expansion for dislocations, which characterize the structure of the discrete elastic far-field.

The equilibrium displacement $\bar{u} \in \mathcal{H}^{1}(\Lambda)$ satisfies
\begin{equation}
  \label{eq:equil_edge}
   \delta \mathcal{E}(\bar{u})[v] = 0 \qquad \forall v \in \mathcal{H}^{\rm c}(\Lambda).
\end{equation}
It is known~\cite{chen19, EOS2016} that \( |D\bar{u}(\ell)| \leq C |\ell|^{-d} \log|\ell| \), showing a slow convergence of standard supercell methods. To improve this, we employ a multipole expansion that captures far-field behavior more accurately and enables the design of higher-order boundary conditions.

In \S~\ref{subsec:atom-model}, we introduced the CLE predictor \( u_0^{\rm C} \), defined by~\eqref{eq:u0} and~\eqref{eq:disl:defn_u0}, and derived from the linearised elasticity theory of dislocations~\cite{anderson2017theory}. This predictor captures the leading-order elastic response of the material in the far field. For \( u : \mathbb{R}^2 \to \mathbb{R}^N \), the CLE model is governed by the continuum energy functional $\mathcal{E}^{\rm C}(u) := \int_{\mathbb{R}^2} W(\nabla u)\, \mathrm{d}x$,
where the energy density \( W \) is defined by the Cauchy--Born rule:
\begin{equation}\label{eq:Wcb}
W(\mathsf{F}) = \frac{1}{|\det A|} V\big((\mathsf{F}\rho)_{\rho \in \mathcal{R}}\big), \qquad \mathsf{F} \in \mathbb{R}^{N \times 2}.
\end{equation}
Linearizing the energy functional around the reference state yields the continuum operator $H^{\rm C} := \delta^2 \mathcal{E}^{\rm C}(0)$,
which plays a central role throughout this work.

To characterize the far-field beyond leading order, we construct higher-order CLE predictors \( u_i^{\rm C} \) for \( i \geq 1 \). Following~\cite{braun2022asymptotic}, these satisfy a hierarchy of elliptic PDEs:
\begin{equation}\label{eq:uCiPDE}
H^{\rm C}[u_i^{\rm C}]= \mathcal{S}_i(\tilde{u}_0^{\rm C}, \dots, \tilde{u}_{i-1}^{\rm C}),
\end{equation}
where \( \mathcal{S}_i \) are nonlinear differential operators depending on the linearized Cauchy--Born tensor \( \mathbb{C} = \nabla^2 W(0) \) and its higher-order derivatives. Here, \( \tilde{u}^{\rm C}_i := u_i^{\rm C} + u_i^{\rm CMP} \) comprises a continuum predictor $u^{\rm C}_i$ and a continuous multipole correction \( u_i^{\rm CMP} \), defined in~\eqref{eq:u_acR} and~\eqref{eq:uiCMP_2}.

A key feature of~\eqref{eq:uCiPDE} is that each \( \mathcal{S}_i \) depends only on \( \tilde{u}_0, \dots, \tilde{u}_{i-1} \) and is independent of \( u_i^{\rm C} \). Consequently, the governing equation for each \( u_i^{\rm C} \) is a linear elliptic problem with a known source term \( \mathcal{S}_i \), allowing recursive computation. The higher-order ($i\geq 1$) CLE predictors satisfy decay estimates
\begin{equation}\label{eq:uCestimate}
\big\lvert \nabla^j u_i^{\rm C} (\ell)\big\rvert \leq C \lvert \ell \rvert^{-j-i} \log^i (\lvert \ell \rvert) \quad\text{for all }j\in\mathbb{N}_0.
\end{equation}

The following theorem characterizes the discrete elastic far-field induced by dislocations using the higher-order CLE predictors introduced above, extending~\cite[Theorem 3.1]{braun2022asymptotic} to edge dislocations. While the statements for screw and edge dislocations are analogous, the edge case involves additional technical challenges arising from lattice mismatch, which we resolve by applying a transformation~\eqref{eq:disl:defn_u0} to correct the CLE predictor \( u^{\rm C}_0 \). A detailed proof is given in Appendix~\ref{app:proof-thm-disloc}.

\begin{theorem}
\label{thm:dislocation}
Suppose that $V \in C^K(\mathbb{R}^{N\times \mathcal{R}}), K\geq J + 2 + p$ with $p\geq 0$ and $J\geq 2$. Let $\mathcal{S}$ be linearly independent with ${\rm span}_{\mathbb{Z}}\mathcal{S} = \Lambda$ and $\mathcal{G}:\Lambda \rightarrow \mathbb{R}^{N\times N}$ be a lattice Green's function defined by~\eqref{eq:def_lattice_G}.
Let $\bar{u} \in \mathcal{H}^1(\Lambda)$ solve \eqref{eq:equil_edge}. 
Then, there exist $u^{\rm C}_i \in C^{\infty}$ and coefficients  $b^{(i,k)}_{\rm exact} \in (\mathbb{R}^{2})^{\odot i}$ such that
\begin{equation}
\label{eq:exp_b}
\bar{u}= \sum_{i=0}^{p} u_i^{\rm C} + \sum_{i=1}^{p} \sum_{k=1}^{N} b^{(i,k)}_{\rm exact} : D^i_\mathcal{S} \mathcal{G}_k + r_{p+1},
\end{equation}
where each $u_i^{\rm C}$ solves the higher-order predictor equation~\eqref{eq:uCiPDE}, and the second summation represents the discrete multipole contributions.
Furthermore, the remainder $r_{p+1}$ satisfies the estimate 
\begin{equation}
\label{eq:maindecay}
\big|D^jr_{p+1}(\ell)\big|\leq |\ell|^{-1-j-p}\log^{p+1}(|\ell|), \quad j = 1,...,J.
\end{equation}
\end{theorem}


\begin{remark}
While the decay estimate~\eqref{eq:maindecay} can be refined to eliminate logarithmic factors~\cite[Remark 7]{braun2022asymptotic}, we retain them to highlight the essential differences from the point defect case~\cite{braun2025higher}, and they have negligible impact on numerical implementation. For point defects, it was shown in~\cite{braun2025higher} that $u_i^{\rm C} = 0$ when $p \leq d$, and the far-field is entirely captured by discrete multipole terms. In contrast, dislocations require the inclusion of higher-order CLE predictors $u_i^{\rm C}$ in the expansion of $\bar{u}$. We employ a precise analytical expression for $u_0^{\rm C}$~\cite{anderson2017theory}. The case $i \geq 1$, however, introduces additional complexity: one must solve elliptic PDEs with singular sources near the core to high accuracy, and the accuracy of $u_i^{\rm C}$ should be carefully matched with that of the multipole coefficients \( b^{(i,k)} \) to ensure optimal error estimate (cf.~Theorem~\ref{th:galerkin}). This continuum–discrete coupling introduces both analytical and computational challenges, and marks a key departure from previous work.
\end{remark}

Theorem~\ref{thm:dislocation} establishes a low-rank decomposition of the discrete elastic far field induced by dislocations. As shown in~\eqref{eq:exp_b}, the equilibrium displacement $\bar{u}$ is expanded into continuum predictors $u_i^{\rm C}$ and discrete multipole terms $b^{(i,k)}_{\rm exact} : D^i_S \mathcal{G}_k$. This representation facilitates precise control over the residual $r_{p+1}$, whose improved regularity and decay underpin the construction of high-order boundary conditions—a central contribution of this work. The coefficients \( b^{(i,k)}_{\rm exact} \) are explicitly approximated from force-moment identities (cf.~\eqref{eq:bIrelation}), introduced in the next section by extending the method of~\cite{braun2025higher} to the dislocation setting.

\section{Accelerated Convergence of Cell Problem and its Numerical Approximations}
\label{sec:Scheme}

In this section, we first revisit a conventional Galerkin scheme for the cell problem and highlight its theoretically accelerated variant. We then introduce a numerical framework consisting of three components: solving higher-order predictor equations~\eqref{eq:uCiPDE}, moment iteration, and its continuous approximation. These form the basis for constructing higher-order boundary conditions with rigorous error control.

\subsection{Accelerated convergence of cell problem}
\label{sec:sub:acc}


Consider the dislocation equilibra stated in Theorem~\ref{thm:dislocation}. We define a family of restricted displacement spaces 
\begin{align}
\mathcal{W}_R &:= \big\{  v : \Lambda \to \mathbb{R}^N \,|\,
                    v(\ell) = 0 \text{ for $|\ell| > R$} \big\}, 
                  \\
\mathcal{U}_R &:= \big\{ u = u_0^{\rm C} + v \,|\, v \in \mathcal{W}_R \big\},
\end{align}
where atoms are clamped in their reference configurations outside a ball of radius $R$. 

We can then approximate \eqref{eq:equil_edge} using the Galerkin scheme: Find $\bar{u}_R \in \mathcal{U}_R$ such that
\begin{equation}
   \label{eq:cellp:galerkin}
   \delta \mathcal{E}(\bar{u}_R)[v] = 0 \qquad \forall v \in \mathcal{W}_R.
\end{equation} 

Under suitable stability conditions, it is demonstrated in~\cite{EOS2016} that for sufficiently large $R$, there exists a solution $\bar{u}_R$ that satisfies the explicit convergence rate  
\begin{equation} \label{eq:slow_convergence_cell}
   \| D\bar{u}_R - D\bar{u} \|_{\ell^2} \leq C R^{-1}\log(R).
\end{equation}
This rate follows directly from the decay estimate $\big|Dr_1(\ell)\big|\lesssim\lvert\ell\rvert^{-2}\log(\lvert\ell\rvert)$ (taking $p=0$ and $j=1$ in \eqref{eq:maindecay}). 



To accelerate the slow convergence of the cell problem \eqref{eq:cellp:galerkin}, we propose an improved far-field boundary condition based on multiple expansions~\cite{braun2022asymptotic}:
\begin{enumerate} 
  \item We replace the naive far-field {\em predictor}
   $\hat{u}_0=u_0^{\rm C}$ with the higher-order continuum {\em predictor}
\begin{equation}
\label{eq:predictor}
\hat{u}_p := \sum_{i = 0}^p u_i^{\rm C},
\end{equation}

where $u_i^{\rm C}$ can be obtained by solving higher-order CLE equations~\eqref{eq:uCiPDE}.

   \item Then, the admissible {\em corrector} space is enlarged with the multipole moments 
  \begin{align}\label{eq:exact_space_CLE}
    \mathcal{U}_{p,R} := \bigg\{
        v : \Lambda \to \mathbb{R}^N \,\Big|\, & \,\,
        v = \sum_{i = 1}^p \sum_{k = 1}^N b^{(i,k)} : D^i_{\mathcal{S}} \mathcal{G}_k
              + w, \nonumber \\[-1em] 
          & \text{for free coefficients $b^{(i,k)}$ and 
          $w \in  \mathcal{W}_R $
           }
    \bigg\},
  \end{align}
  where the {\em corrector} displacement is parameterised by its values in the domain $\Lambda \cap B_R$ and by the discrete coefficients $b^{(i,k)}$ of the multipole terms.
  \item We consider the pure Galerkin approximation: Find $v^*_{p,R}\in\mathcal{U}_{p,R}$, $u^*_{p,R} = \hat{u}_p + v^*_{p,R}$, such that 
\begin{equation} \label{eq:galerkin}
    \delta \mathcal{E}(u^{*}_{p, R}) [v_R] = 0 \qquad \forall v_R \in \mathcal{W}_R.   
\end{equation}
\end{enumerate}



The following theorem outlines the error estimates for the Galerkin approximation~\eqref{eq:galerkin}, focusing on both geometric error and energy error. We refer to~\cite[Section 7.2]{EOS2016} for the proof. It is important to compare the improved rate with the rate of the naive scheme~\eqref{eq:slow_convergence_cell} for $p\geq 1$.  

\begin{theorem}\label{th:galerkin}
  Suppose that $\bar{u}$ is a strongly stable solution of \eqref{eq:equil_edge}; that is, there exists a stability constant $c_0 > 0$ such that 
\begin{equation}
\label{eq:stabeq}
\delta^2 \mathcal{E}(\bar{u})[v,v] \geq c_0 \| Dv \|^2_{\ell^2}, \qquad \forall v \in \mathcal{H}^1(\Lambda),
\end{equation}
  then, for $R$ sufficiently large, there exists a solution $v^*_{p,R} \in \mathcal{U}_{p,R}$, $u^*_{p,R} = \hat{u}_p + v^*_{p,R}$ to~\eqref{eq:galerkin} with $b^{(i,k)} = b^{(i,k)}_{\rm exact}$ and such that 
\begin{align}\label{eq:thm3.1-1}
 \big\| D\bar{u} - D\bar{u}^{*}_{p,R} \big\|_{\ell^2}
 &\leq C_{\rm G} 
 R^{- 1 - p} \log^{p+1}(R), \\
 \big|\mathcal{E}(\bar{u})-\mathcal{E}(\bar{u}^*_{p, R}) \big| &\leq C_{\rm E} 
 R^{- 2 - 2p} \log^{2p+2}(R). 
\end{align}
\end{theorem}

The foregoing theorem is an important theoretical milestone, showcasing the accelerated convergence of cell problems that can in principle be achieved. However, the direct implementation of the scheme~\eqref{eq:galerkin} is not feasible, and numerical approximations must inevitably be introduced. In particular, the predictor equation must be solved numerically to approximate the far-field predictor. Special care is also required in handling the terms $b^{(i,k)}$, which should be approximated on a finite domain. Moreover, for computational efficiency, it is necessary to adopt a continuous formulation for their evaluation. We will provide a detailed numerical framework addressing these issues in the following section.

\subsection{Numerical approximations of the cell problems}
\label{sec:sub:approximation}

In this section, we introduce three numerical approximations involved in solving the cell problems~\eqref{eq:galerkin}, and analyze the associated errors. Our main result, Theorem~\ref{th:galerkinaM}, demonstrates that the numerical scheme achieves the same order of convergence as the theoretical result in Theorem~\ref{th:galerkin}, up to controllable truncation and numerical discretization errors.

\subsubsection{Predictor approximations}
\label{sec:sub:sub:pde}

The first step is to compute the higher-order continuum predictor \( \hat{u}_p \) for each \( u_i^{\rm C} \), \( i \geq 1 \), which satisfies a sequence of second-order elliptic equations~\eqref{eq:uCiPDE} posed on an infinite domain. To make this numerically tractable, we truncate the domain, introducing a controllable truncation error, and then discretize the equation, which incurs numerical error. 

To that end, an important observation that justifies the truncation of the infinite domain is the decay estimate for the predictor \( u_i^{\rm C} \) (cf.~\eqref{eq:uCestimate}).
This allows us to restrict the computation to a finite domain of radius \( R_{\rm c} \) with controllable error (see \eqref{eq:approx-pde} and Lemma~\ref{thm:pde} for details). 

After truncating to a finite domain, a remaining challenge is the singular behavior of \( u_i^{\rm C} \) near the dislocation core, which remains the primary obstacle to solving the higher-order predictor equations. To overcome this, we introduce the rescaled form \( u_i^{\rm C} = v_i \cdot r^i \), transforming the original problem into equations for \( v_i \). These are then solved using a spectral Galerkin method~\cite{2011Spectral}, as detailed in \S~\ref{sec:sub:pde}.

With these approximations in place, we obtain a numerical approximation of the higher-order predictor, denoted by \( \hat{u}^{\rm num}_{p} \). This can be incorporated into the following Galerkin scheme: find \( v^*_{p,R} \in \mathcal{U}_{p,R} \) such that \( \bar{u}^{\rm c}_{p,R} := \hat{u}^{\rm num}_{p} + v^*_{p,R} \) satisfies
\begin{equation}
\delta \mathcal{E}(\bar{u}^{\rm c}_{p,R}) [v_R] = 0 \qquad \forall v_R \in \mathcal{W}_{R}.
\end{equation}

From~\eqref{eq:predictor}, it follows that $\hat{u}^{\rm num}_{p}$ is obtained by summing $u_i^{\mathrm{C}}$ for $i = 0, \dots, p$. We now present an error estimate for the approximation of the continuous predictor solution $\bar{u}^{\rm c}_{p,R}$ by its numerical counterpart $u^*_{p,R}$, which is obtained by solving the elliptic equations~\eqref{eq:uCiPDE} numerically. This estimate follows directly from the error bounds for each $u_i^{\mathrm{C}}$ established in Lemma~\ref{thm:pde} in~\S~\ref{sec:sub:sub:spectral}:
\begin{equation}
\label{eq:approx-pde}
\big\| D u^*_{p,R} - D\bar{u}^{\rm c}_{p,R} \big\|_{\ell^2} \lesssim C_p \left(R_{\rm c}^{-1} \log(R_{\rm c}) + e^{-c{N_{\rm pde}}}\right),
\end{equation}
where \( N_{\rm pde} \) denotes the number of degrees of freedom used to solve the predictor equations, and \( C_p > 0 \) is a constant independent of \( N_{\rm pde} \).

From the error estimate~\eqref{eq:approx-pde}, we observe that once a sufficiently large computational domain radius \( R_{\rm c} \) is chosen, the truncation error becomes negligible. Since the predictor equations are solved only once, the associated computational cost is minimal—this is further supported by our numerical efficiency analysis in \S~\ref{sec:sub:conv}, which shows that this cost is negligible compared to solving the cell problem. Moreover, thanks to the spectral accuracy of the numerical scheme, we can achieve highly accurate approximations of the higher-order continuum predictors. This capability forms a central motivation for the methodology developed in this work.

\subsubsection{Multipole moment approximations}
\label{sec:sub:moment}

Next, we approximate the multipole moment terms \( b^{(i,k)}_{\rm exact} \), which are originally defined on the infinite domain. To make this tractable, we adopt the moment iteration strategy developed for point defects~\cite[Sec. 3]{braun2025higher} and adapt it to our setting. Specifically, we fix the multipole tensors in the corrector space, resulting in the following approximate corrector space:
\begin{align}
\label{eq:u_bcR}
\mathcal{U}_{p,R}^{\rm (b)} := \bigg\{
v : \Lambda \to \mathbb{R}^N \,\bigg|\, \,\,
v = \sum_{i = 1}^p &\sum_{k = 1}^N b^{(i,k)} : D^i_{\mathcal{S}} \mathcal{G}_k
      + w, \nonumber \\[-0.5em] 
 & \text{for fixed } b \text{ and } w \in \mathcal{W}_R 
\bigg\}.
\end{align}
Here, the tensors \( b^{(i,k)} \) are chosen as approximations to the exact moments \( b^{(i,k)}_{\rm exact} \), which involves evaluating the multipole tensors within a finite domain.

We now consider the corresponding Galerkin approximation: find \( \bar{v}_{p,R}^{\rm b} \in \mathcal{U}^{({\rm b})}_{p,R} \), set \( \bar{u}_{p,R}^{\rm b} = \hat{u}^{\rm num}_{p} + \bar{v}_{p,R}^{\rm b} \), such that
\begin{equation} \label{eq:galerkin_approx}
    \delta \mathcal{E}(\bar{u}_{p,R}^{\rm b}) [v_{R}] = 0 \qquad \forall v_{R} \in \mathcal{W}_{R}.
\end{equation}
The resulting approximation error arises solely from the moment iteration used to estimate \( b^{(i,k)}_{\rm exact} \). While we omit a detailed discussion here, the iterative scheme and supporting analysis are provided in Appendix~\ref{subsec:moment iteration}. For completeness, we state the resulting error estimate:
\begin{equation}
\label{eq:approx-b-R}
\big\| D \bar{u}_{p,R}^{\rm c}  - D \bar{u}_{p,R}^{\rm b} \big\|_{\ell^2}
\lesssim R^{-1-p}\log^{p+1}(R).
\end{equation}

\subsubsection{Continuous multipole expansion}
\label{sec:sub:coeffts}

To avoid the complexity of discrete Green's functions and their derivatives, we adopt a continuous approximation proposed in~\cite{braun2022asymptotic}. By leveraging the connection between the discrete Green's function \( \mathcal{G} \) and its continuum counterpart \( G \), we derive the continuous multipole moment tensors, which define the following approximate corrector space:
\begin{align}
\label{eq:u_acR}
\mathcal{U}_{p,R}^{\rm (a)} := \bigg\{
v : \Lambda \to \mathbb{R}^N \,\bigg|\, \,\,
v = \sum_{k = 1}^N&\sum_{n=0}^{\lfloor\frac{p-1}{2}\rfloor}\sum_{i = 1}^{p-2n} a^{(i,n,k)} : \nabla^i(G)_{\cdot k}
+ w, \nonumber \\[-0.5em] 
 & \text{for fixed } a \text{ and } w \in \mathcal{W}_R 
\bigg\},
\end{align}
where the first term, corresponding to the continuous multipole (CMP), also enters the construction of $\tilde{u}_i^{\rm C}$ in~\eqref{eq:uCiPDE}.

We then consider the corresponding Galerkin scheme: find \( \bar{v}_{p,R} \in \mathcal{U}^{\rm (a)}_{p,R} \) and set \( \bar{u}_{p,R} = \hat{u}^{\rm num}_{p} + \bar{v}_{p,R} \) such that
\begin{equation} \label{eq:galerkin_approx_a}
\delta \mathcal{E}(\bar{u}_{p,R}) [v_{R}] = 0 \qquad \forall v_{R} \in \mathcal{W}_{R}. 
\end{equation}
The resulting approximation error from replacing the discrete multipole expansion with its continuous counterpart is estimated as (see Appendix~\ref{app:sub:a} for details)
\begin{equation}
\label{eq:approx-a}
\big\| D \bar{u}_{p,R}^{\rm b} - D \bar{u}_{p,R} \big\|_{\ell^2} \lesssim R^{-1-p}.
\end{equation}

We now rigorously quantify the three sources of numerical approximation errors discussed above. The following theorem establishes convergence estimates for the Galerkin approximation~\eqref{eq:galerkin_approx_a}, measured in both the discrete energy norm and the geometry error norm. This theorem quantifies the total error resulting from the combined effect of the three approximation steps described above. A detailed proof is provided in Appendix~\ref{app:proof-of-thm3.2}.

\begin{theorem}\label{th:galerkinaM}
Suppose that $\bar{u}$ is a strongly stable solution of \eqref{eq:equil_edge}, 
for sufficiently large radius $R$, there exists a corrector $\bar{u}_{p,R} = \hat{u}^{\rm num}_{p} + \bar{v}_{p,R}$ with $\bar{v}_{p,R} \in \mathcal{U}_{p,R}^{\rm (a)}$ solving \eqref{eq:galerkin_approx_a} and $C_g = {\rm max}\{C_G, C_p\}$, $C_e = {\rm max}\{C_E, C_p\}$ such that 
 \begin{align}
     \lVert D\bar{u} - D \bar{u}_{p,R} \rVert_{\ell^2}
     &\leq C_g \big(R^{-1-p}\log^{p+1}(R)+ R_{\rm c}^{-1} \log(R_{\rm c}) +e^{-cN_{\rm pde}}\big), \qquad \text{and} \\
     \big|\mathcal{E}(\bar{u})-\mathcal{E}( \bar{u}_{p,R} )\big| &\leq C_e\big(R^{-2-2p} \log^{2p+2}(R)+ R_{\rm c}^{-1} \log(R_{\rm c}) + e^{-cN_{\rm pde}}\big).
 \end{align}

\end{theorem}

This theorem establishes that the numerical scheme retains the same convergence rate as the theoretical result in Theorem~\ref{th:galerkin}, up to controllable truncation and discretization errors. Numerical validation will be presented in \S~\ref{sec:numexp}.

\section{Numerical Algorithms}
\label{sec:pde}

In this section, we provide a more detailed discussion of the numerical solution of the higher-order continuum equations, along with the construction of the main algorithm (cf.~Algorithm~\ref{alg:moment_iter_a}).

\subsection{Solving higher-order predictor equations}
\label{sec:sub:pde}

To solve the higher-order predictor equations~\eqref{eq:uCiPDE}, we construct the modified continuum predictor
\begin{equation}
\label{eq:uiCMP_2}
\tilde{u}_i^{\rm C} := u_i^{\rm C} + u_i^{\rm CMP}, \qquad
u_i^{\rm CMP} := \sum_{k = 1}^N \sum_{n = 0}^{\lfloor \frac{p - 1}{2} \rfloor} \sum_{j = 1}^{p - 2n} a^{(j,n,k)} : \nabla^j(G)_{\cdot k},
\end{equation}
where \( u_i^{\rm C} \) is the solution of the higher-order CLE problem for \( i \geq 1 \), and the correction term \( u_i^{\rm CMP} \) is constructed from the CMP basis~\eqref{eq:u_acR}.

The right-hand sides (i.e., the higher-order residual forces) of the first three equations of~\eqref{eq:uCiPDE} (derived in detail in~\cite[Section 7]{braun2022asymptotic}) are given by
\begin{equation}
\begin{aligned}
\mathcal{S}_0 &= 0, \\
\mathcal{S}_1(u_0^{\rm C}) &= \frac{1}{2} \divo\big(\nabla^3 W(0) [\nabla u_0^{\rm C}]^2 \big), \\
\mathcal{S}_2(\tilde{u}_0^{\rm C}, \tilde{u}_1^{\rm C}) &= \divo\big(\nabla^3 W(0) [\nabla u_0^{\rm C}, \nabla u_1^{\rm C} + \nabla u_1^{\rm CMP}] \big) \\
&\quad + \frac{1}{6} \divo\big(\nabla^4 W(0) [\nabla u_0^{\rm C}]^3 \big) - \widetilde{H}[u_0^{\rm C}],
\end{aligned}
\end{equation}
where the nonlocal correction operator \( \widetilde{H} \) is defined as
\begin{equation*}
\widetilde{H}[u] := \frac{1}{12 c_{\rm vol}} \sum_{\sigma, \rho \in \mathcal{R}} \nabla^2 V(0)_{\sigma \rho} \Big(
  3 \nabla^4 u[\sigma, \sigma, \rho, \rho]
  - 2 \nabla^4 u[\sigma, \rho, \rho, \rho]
  - 2 \nabla^4 u[\sigma, \sigma, \sigma, \rho]
\Big),
\end{equation*}
with $c_{\rm vol}:=|\det(A)|$. We restrict attention to these first three equations, as they suffice for our numerical implementation. Higher-order terms follow similar structure and are omitted for brevity.

Recall the finite domain approximation (cf.~Sec.~\ref{sec:sub:sub:pde}) and the corrector equations~\eqref{eq:uCiPDE}, the equations for \( u^{\rm C}_i \) ($i=1,2$) can be explicitly written as
\begin{equation}
\label{u1-pde}
\begin{aligned}
\begin{cases}
-\divo(\mathbb{C}[\nabla u_i^{\rm C}]) = g_i & \text{in } B_{R_{\rm c}}, \\
u_i^{\rm C} = 0 & \text{on } \partial B_{R_{\rm c}},
\end{cases}
\end{aligned}
\end{equation}
where the right-hand sides \( g_i \) are given by
\begin{equation}
\label{eq:gi}
g_i =
\begin{cases}
\displaystyle
\frac{1}{2} \divo\big( \nabla^3 W(0)[\nabla u_0^{\rm C}, \nabla u_0^{\rm C}] \big), & i = 1, \\[1.2em]
\displaystyle
\divo\big( \nabla^3 W(0)[\nabla u_0^{\rm C}, \nabla \tilde{u}_1^{\rm C}] \big)
+ \frac{1}{6} \divo\big( \nabla^4 W(0)[\nabla u_0^{\rm C}]^3 \big) 
- \widetilde{H}[u_0^{\rm C}], & i = 2,
\end{cases}
\end{equation}
and \( W(0) \) denotes the Cauchy--Born energy density evaluated at the reference configuration.
We remark that the system~\eqref{u1-pde} is solved with homogeneous Dirichlet boundary conditions for simplicity. Alternative choices, such as periodic boundary conditions, are also possible and have been explored in~\cite{cai2003periodic}, though they require additional technical treatment. 

\subsubsection{Removing the singularities at $r=0$}
\label{sec:remove-singularity}

To solve~\eqref{u1-pde} with high accuracy, we adopt a spectral method in polar coordinates over the computational domain \( B_{R_{\rm c}} \), which we denote as \( \Omega_{\rm c} := (0, R_{\rm c}) \times [0, 2\pi) \). Since the PDE form is preserved under the coordinate transformation, we retain the notation \( u_i^{\rm C} \) and \( u_i^{\rm CMP} \) in polar coordinates. The displacement fields \( u_i^{\rm C} \) satisfy the decay estimate of~\eqref{eq:uCestimate} in polar coordinates, i.e., $|\nabla^j u^{\rm C}_i(r)| \lesssim C |r|^{-i-j} \log |r|$.

To address the singularity at $r = 0$, we introduce the rescaled variables
\begin{equation}
\label{eq:rescale}
\left\{
\begin{aligned}
v_1 &= r \cdot u_1^{\rm C}, \\
v_0 &= r \cdot \nabla u_0^{\rm C},
\end{aligned}
\right. ~~{\rm and}~~
\left\{
\begin{aligned}
v_2 &= r^2 \cdot u_2^{\rm C}, \\
v_1 &= r \cdot u_1^{\rm C}, \\
v_1^{\rm CMP} &= r \cdot u_1^{\rm CMP}.
\end{aligned}
\right.
\end{equation}
Expanding the divergence terms in~\eqref{u1-pde} and \eqref{eq:gi} yields the following general form:
\begin{equation} 
\label{eq:u1-pde-r}
\begin{aligned}
-r^2 \divo\left( \mathbb{C} \nabla v_i \right) + a_ir\mathbb{C}\nabla v_i  \hat{r} + b_i\mathbb{C}v_i\hat{r}^2  &= f_i ,&(r,\theta) \in \Omega_{\rm c}, 
\end{aligned}
\end{equation}
where $a_1 = 2, a_2 = 4,b_1 = -2, b_2 = -6$, $\hat{r}$ is the unit vector in polar coordinates and 
\begin{equation}
\label{eq:fi_cases}
f_i = 
\begin{cases}
\displaystyle
\frac{r}{2} \divo\big( \nabla^3 W(0)[v_0, v_0] \big)
- \nabla^3 W(0)[v_0, v_0] \, \hat{r}, & i = 1, \\[1.2em]

\displaystyle
r^2 \divo\big( \nabla^3 W(0)[v_0, \nabla v_1 + \nabla v_1^{\rm CMP}] \big)
- r \divo\big( \nabla^3 W(0)[v_0, v_1 + v_1^{\rm CMP}] \big) \\[0.5em]
\displaystyle\quad
+ \frac{r}{6} \divo\big( \nabla^4 W(0)[v_0]^3 \big)
- 2r \nabla^3 W(0)[v_0, \nabla v_0 + \nabla v_1^{\rm CMP}] \, \hat{r} \\[0.5em]
\displaystyle\quad
- \frac{1}{2} \nabla^4 W(0)[v_0]^3 \, \hat{r}
+ 3 \nabla^3 W(0)[v_0, v_1 + v_1^{\rm CMP}] \, \hat{r}^2
+ r^4 \widehat{H}[v_0], & i = 2.
\end{cases},
\end{equation}
where $\widehat{H}[v_0]$ is derived from $\widetilde{H}[u_0^{\rm C}]$ in~\eqref{eq:gi} via the transformation in~\eqref{eq:rescale}.


Based on the analytic regularity theory for linear elliptic systems~\cite{morrey1957analyticity}, we establish the following result for the solution \( v_i \) to~\eqref{eq:u1-pde-r}:
\begin{lemma}
\label{lem:analytic_solution}
Let the coefficient tensor \( \mathbb{C} = \nabla^2 W(0) \) be positive definite, and assume the source term \( f_i \) is analytic in \( \Omega_{\rm c} \). Then the system~\eqref{eq:u1-pde-r} admits a unique solution \( v_i \), which is analytic in \( \Omega_{\rm c} \).
\end{lemma}

The analyticity of \( v_i \) allows the application of spectral Galerkin methods to solve~\eqref{eq:u1-pde-r} with high-order accuracy.

\subsubsection{Spectral Galerkin approximation}
\label{sec:sub:sub:spectral}

For a cut-off parameter $M > 0$, let $\theta_j = \{j\pi/M\}_{j=0}^{2M-1}$ denote the standard Fourier collocation points. We expand the source term $f_i$ and the solution $v_i$ in truncated Fourier series as
\begin{equation}
\label{eq:fourier-vi-fi}
\begin{aligned}
f_{i,M}(r, \theta_j) &= \sum_{m=0}^{M} \big( f_{i,1m}(r) \cos(m \theta_j) + f_{i,2m}(r) \sin(m \theta_j) \big), \\[-0.2em]
v_{i,M}(r, \theta_j) &= \sum_{m=0}^{M} \big( v_{i,1m}(r) \cos(m \theta_j) + v_{i,2m}(r) \sin(m \theta_j) \big).
\end{aligned}
\end{equation}

Substituting the above expansion into \eqref{eq:u1-pde-r}, and collecting the cosine and sine modes, we obtain for each $m = 0, 1, \dots, M$:
\begin{equation}
\label{eq:v1-in-pole}
\begin{aligned}
-r^2\widetilde{C}_1 v_{i,1m}'' + \widetilde{C}_1 r v_{i,1m}' - \widetilde{C}_2 r v_{i,2m}' - \widetilde{C}_3 v_{i,1m} + \widetilde{C}_2 v_{i,2m} &= f_{i,1m}, \\
-r^2\widetilde{C}_1 v_{i,2m}'' + \widetilde{C}_1 r v_{i,2m}' + \widetilde{C}_2 r v_{i,1m}' - \widetilde{C}_3 v_{i,2m} - \widetilde{C}_2 v_{i,1m} &= f_{i,2m},
\end{aligned}
\end{equation}
subject to the boundary conditions:
\[
v_{i,1m}(0) = v_{i,2m}(0) = 0 \quad \text{if } m \neq 0, \qquad
v_{i,1m}(R) = v_{i,2m}(R) = 0.
\]
In line with~\cite{2011Spectral}, no additional pole conditions are imposed at \( m = 0 \), as their implementation is nontrivial and may reduce accuracy in certain cases (e.g., when the exact solution depends on \( r - 1 \)). The constant matrices \( \widetilde{C}_1, \widetilde{C}_2, \widetilde{C}_3 \) are defined as
\begin{equation*}
\widetilde{C}_1 = \begin{bmatrix}
C_{11} & 0 \\
0 & C_{44}
\end{bmatrix}, \quad
\widetilde{C}_2 = m(C_{12} + C_{44}) I, \quad
\widetilde{C}_3 = (1 - m^2)(C_{11} + C_{44}) I,
\end{equation*}
where \( C_{11}, C_{12}, C_{44} \) are the three independent elastic constants of the cubic crystal system~\cite{haussuhl2008physical}, and are components of the elasticity tensor \( \mathbb{C} \).

Following~\cite{2011Spectral}, we apply the coordinate transformation \( r = (R_{\rm c} + \zeta)/2 \) in~\eqref{eq:v1-in-pole}, mapping the domain to \( D = (-R_{\rm c}, R_{\rm c}) \). For notational convenience, we define \( v(\zeta) := v((R_{\rm c} + \zeta)/2) \) and \( g_i(\zeta) := 4 f_i((R_{\rm c} + \zeta)/2) \) for \( i = 1,2 \), resulting in
\begin{equation}
\label{eq:coordinate transformation}
\begin{aligned}
-\widetilde{C}_1(R_{\rm c}+\zeta)^2\frac{\partial^{2}v_1}{\partial \zeta^{2}} + 2\widetilde{C}_1(R_{\rm c}+\zeta)\frac{\partial v_1}{\partial \zeta} -2\widetilde{C}_2(R_{\rm c}+\zeta)\frac{\partial v_2}{\partial \zeta} -  2\widetilde{C}_3v_1 + 2\widetilde{C}_2v_2 &= g_1, \\
-\widetilde{C}_1(R_{\rm c}+\zeta)^2\frac{\partial^{2}v_2}{\partial \zeta^{2}} + 2\widetilde{C}_1(R_{\rm c}+\zeta)\frac{\partial v_2}{\partial \zeta} +2\widetilde{C}_2(R_{\rm c}+\zeta)\frac{\partial v_1}{\partial \zeta} -  2\widetilde{C}_3v_2 - 2\widetilde{C}_2v_1 &= g_2, \\
v_1(-R_{\rm c}) = v_2(-R_{\rm c}) = 0 \text{ if } m \neq 0,\quad v_1(R_{\rm c})=v_2(R_{\rm c})  &= 0.
\end{aligned}
\end{equation}
We now analyze the cases $m \neq 0$ and $m = 0$ separately.

{\bf Case $\bm{m\neq 0}$}: Let $L_k(t)$ is the $k$th degree Legendre polynomial and define
\begin{equation*}
X_{N_{\rm pde}}(m) = \text{span}\left\{\phi_i(t) = L_i(\zeta/R_{\rm c}) - L_{i+2}(\zeta/R_{\rm c}), i=0,1,...,{N_{\rm pde}}-2\right\}.
\end{equation*}
It is evident that $v(\pm R_{\rm c}) = 0$ holds for all $v\in X_{N_{\rm pde}}(m)$. Then the spectral Legendre-Galerkin approximation to \eqref{eq:coordinate transformation} is to find $v_{1,{N_{\rm pde}}}, v_{2,{N_{\rm pde}}}\in X_{N_{\rm pde}}(m)$ such that
\begin{equation}
\label{eq:weighted weak formulation}
\begin{aligned}
&\left\{
\begin{aligned}
&\left( \widetilde{C}_1(R_{\rm c}+\zeta)^2 v'_{1,{N_{\rm pde}}}, w'_1 \right) + \left( 2\widetilde{C}_1(R_{\rm c}+\zeta) v'_{1,{N_{\rm pde}}}, w_1 \right) - \left( 2\widetilde{C}_3 v_{1,{N_{\rm pde}}}, w_1 \right)   \\
&\quad - \left( 2\widetilde{C}_2(R_{\rm c}+\zeta) v'_{2,{N_{\rm pde}}}, w_1 \right) + \left( 2\widetilde{C}_2 v_{2,{N_{\rm pde}}}, w_1 \right) \\
&\quad = \left( I_{N_{\rm pde}} g_1, w_1 \right), \quad \forall w_1 \in X_{N_{\rm pde}}(m), \\
&\left( \widetilde{C}_1(R_{\rm c}+\zeta)^2 v'_{2,{N_{\rm pde}}}, w'_2 \right) + \left( 2\widetilde{C}_1(R_{\rm c}+\zeta) v'_{2,{N_{\rm pde}}}, w_2 \right) - \left( 2\widetilde{C}_3 v_{2,{N_{\rm pde}}}, w_2 \right) \\
&\quad + \left( 2\widetilde{C}_2(R_{\rm c}+\zeta) v'_{1,{N_{\rm pde}}}, w_2 \right) - \left( 2\widetilde{C}_2 v_{1,{N_{\rm pde}}}, w_2 \right) \\
& \quad = \left( I_{N_{\rm pde}} g_2, w_2 \right), \quad \forall w_2 \in X_{N_{\rm pde}}(m).
\end{aligned}
\right.
\end{aligned}
\end{equation}
where $(\cdot,\cdot)$ is the $L^2$-inner product in $D$ and $I_{N_{\rm pde}}$ is the interpolation operator relative to the Gauss-Lobatto points. We define

\begin{alignat*}{3}
\varphi_{ij} &= \int_I (R_{\rm c} + \zeta)^2 \, \phi_j'(\zeta) \, \phi_i'(\zeta) \, \mathrm{d}\zeta, 
&\enspace \psi_{ij} &= \int_I (R_{\rm c} + \zeta) \, \phi_j'(\zeta) \, \phi_i(\zeta) \, \mathrm{d}\zeta, 
&\enspace \mu_{ij} &= \int_I \phi_j \phi_i \, \mathrm{d}\zeta, \\
f_{1,i}     &= \int_I (I_{N_{\rm pde}} g_1)(\zeta) \, \phi_i(\zeta) \, \mathrm{d}\zeta, 
&\enspace f_{2,i} &= \int_I (I_{N_{\rm pde}} g_2)(\zeta) \, \phi_i(\zeta) \, \mathrm{d}\zeta, 
\end{alignat*}
and
\begin{equation*}
v_{1,{N_{\rm pde}}} = \sum_{i=0}^{N_{\rm pde}-2} x_i \, \phi_i(\zeta), \enspace v_{2,{N_{\rm pde}}} = \sum_{i=0}^{N_{\rm pde}-2} y_i \, \phi_i(\zeta).
\end{equation*}
For \(i,j = 0,1,\ldots,N_{\rm pde}-2\), define the matrices \(\boldsymbol{\varphi} = (\varphi_{ij})\), \(\boldsymbol{\psi} = (\psi_{ij})\), and \(\boldsymbol{\mu} = (\mu_{ij})\). Let the vectors \(\mathbf{f}_1 = (f_{1,0}, \ldots, f_{1,N_{\rm pde}-2})^\top\), \(\mathbf{f}_2 = (f_{2,0}, \ldots, f_{2,N_{\rm pde}-2})^\top\), \(\mathbf{x} = (x_0, \ldots, x_{N_{\rm pde}-2})^\top\), and \(\mathbf{y} = (y_0, \ldots, y_{N_{\rm pde}-2})^\top\). Then, the weak formulation~\eqref{eq:weighted weak formulation} reduces to
\begin{equation}
\label{eq:linear system}
\begin{bmatrix}
\widetilde{C}_1 \varphi + 2\widetilde{C}_1 \psi - 2\widetilde{C}_3 \mu & -2\widetilde{C}_2 \psi + 2\widetilde{C}_2 \mu \\
2\widetilde{C}_2 \psi - 2\widetilde{C}_2 \mu & \widetilde{C}_1 \varphi + 2\widetilde{C}_1 \psi - 2\widetilde{C}_3 \mu
\end{bmatrix}
\begin{bmatrix}
\mathbf{x} \\
\mathbf{y}
\end{bmatrix}
=
\begin{bmatrix}
\mathbf{f}_1 \\
\mathbf{f}_2
\end{bmatrix}.
\end{equation}

{\bf Case $\bm{m = 0}$}:  In this case, we define
\begin{equation*}
X_{N_{\rm pde}}(0) := \text{span}\left\{ \phi_i(\zeta) = L_i\left(\frac{\zeta}{R_{\rm c}}\right) - L_{i+1}\left(\frac{\zeta}{R_{\rm c}}\right) \,\Big|\, i = 0, 1, \ldots, N_{\rm pde}-1 \right\}.
\end{equation*}
It is straightforward to verify that \( v(R) = 0 \) for all \( v \in X_{N_{\rm pde}}(0) \).
Similarly, extending the index ranges in~\eqref{eq:linear system} to \(i,j = 0, \ldots, {N_{\rm pde}}-1\), we obtain
\begin{equation}
\label{eq:linear system-m=0}
\begin{bmatrix}
\widetilde{C}_1 \varphi + 2\widetilde{C}_1 \psi - 2\widetilde{C}_3 \mu & 0 \\
0 & \widetilde{C}_1 \varphi + 2\widetilde{C}_1 \psi - 2\widetilde{C}_3 \mu
\end{bmatrix}
\begin{bmatrix}
\mathbf{x} \\
\mathbf{y}
\end{bmatrix}
=
\begin{bmatrix}
\mathbf{f}_1 \\
\mathbf{f}_2
\end{bmatrix}.
\end{equation}

Building on the singularity removal procedure described in~\S~\ref{sec:remove-singularity}, we express each $u_i$ in terms of its analytic counterpart $v_i$. The spectral Galerkin approximation of $v_i$ is then computed following the construction in~\S~\ref{sec:sub:sub:spectral}, yielding the components $v_{i,1m,N_{\rm pde}}$ and $v_{i,2m,N_{\rm pde}}$. These approximations allow us to reconstruct $u_i^{\rm C}$ as
\[
u_{i,{N_{\rm pde}}}^{\rm C} = r\cdot\sum_{m=0}^{M} \big( v_{i,1m,{N_{\rm pde}}}(r) \cos(m \theta_j) + v_{i,2m,{N_{\rm pde}}}(r) \sin(m \theta_j) \big).
\]
The following lemma, whose proof is given in Appendix~\ref{app:proof-thm-pde}, provides an error estimate for the spectral approximation. While our discussion in \S~\ref{sec:pde} focuses on \(i = 1, 2\), the result extends to general \(i\), as all cases reduce to second-order elliptic problems whose approximation error depends primarily on the solution regularity.

\begin{lemma}
\label{thm:pde}
Under the condition of Lemma~\ref{lem:analytic_solution}, let \( u_i^{\rm C} \) be the solutions to~\eqref{eq:uCiPDE} for $i\geq 1$. Then
there exist positive constants \( C_{\rm s} \) and $c$, independent of  \( N_{\rm pde} \) and \( R_{\rm c} \), such that 
\begin{equation}
\label{eq:pde-error}
\| \nabla u_i^{\rm C} - \nabla u_{i,{N_{\rm pde}}}^{\rm C} \|_{L^2} \leq C_{\rm s} \left( R_{\rm c}^{-i} \log(R_{\rm c}) + e^{-c{N_{\rm pde}}} \right).
\end{equation} 
\end{lemma}

\subsection{Numerical algorithm for higher-order boundary conditions}

We now present the main algorithm for constructing higher-order boundary conditions, building upon the preceding analytical framework and approximation strategies.

To be consistent with the approximate corrector space defined in~\eqref{eq:u_acR}, we reformulate the equilibrium condition~\eqref{eq:exp_b} into its continuous counterpart. Following the analysis in~\cite[Lemma 3.5]{braun2025higher}, and under the assumptions of Theorem~\ref{thm:dislocation}, for $p = 1, 2$, there exist tensors $a^{(i,0,k)} \in (\mathbb{R}^{d})^{\odot i}$ for $1 \leq i \leq p$, $1 \leq k \leq N$ such that
\begin{equation}
\label{eq:exp_ai}
\bar{u} = \sum_{i=0}^{p} u_i^{\rm C} + \sum_{k=1}^{N} \left( \sum_{i=1}^{p} a^{(i,0,k)} : \nabla^i (G_0)_{\cdot k} \right) + w,
\end{equation}
where $w$ denotes the asymptotic remainder. Furthermore, for all $j = 1, 2$ and $\alpha \in \mathbb{N}_0$, the remainder satisfies the decay estimate
\begin{equation}
\label{eq:structurewithmomentscontremainder}
| D^j w(\ell) | \lesssim |\ell|^{-2-j} \log^{\alpha + 1}(|\ell|),
\end{equation}
which ensures fast spatial decay of higher-order corrections away from the defect core.

A direct correspondence exists between the coefficients $a^{(i,0,k)}$ and the force moments of the displacement field. Specifically, for $i = 1, 2$, collecting the components of $a^{(i,0,k)}$ over all sites and directions yields:
\begin{equation}
\label{eq:a-I_}
a^{(1,0,\cdot)} = -\mathcal{I}_1[\bar{u}], \quad a^{(2,0,\cdot)} = \tfrac{1}{2} \mathcal{I}_2[\bar{u}],
\end{equation}
where $\mathcal{I}_i[\bar{u}]$ denotes the $i$-th order force moment. In practice, these moments are approximated by their truncated versions $\mathcal{I}_{i,R}$, computed over a finite domain (cf.~\eqref{eq:defn_IjR}). This leads naturally to the \emph{moment iteration scheme} described in Appendix~\ref{app:sec:b}, and detailed in~\cite[Algorithm 3.1]{braun2025higher}.

Together, these results constitute a complete computational framework for constructing higher-order boundary conditions via force moment. For further implementation details and theoretical justifications, we refer the reader to~\cite[Section 3.4]{braun2025higher}.

\begin{algorithm}[!htb]
\caption{Computation of {\em correctors} with higher-order boundary condition}
\label{alg:moment_iter_a}
\begin{enumerate}
\item Compute {\em the zeroth-order corrector}: $\bar{u}_{0, R}=\bar{u}_R$ such that \eqref{eq:cellp:galerkin} holds. 
The convergence $\lVert D\bar{u} - D\bar{u}_{0, R} \rVert_{\ell^2} \lesssim R^{-1}\log(R)$, $\big|\mathcal{E}(\bar{u})-\mathcal{E}(\bar{u}_{0, R})\big| \lesssim R^{-2} \cdot \log^{2}(R)$ are then obtained (cf.~$p=0$ in Theorem~\ref{th:galerkinaM}).

\item Solve the higher-order predictor equations~\eqref{eq:uCiPDE} for $u_1^{\rm C}$ using the spectral Galerkin method in~\S~\ref{sec:sub:sub:spectral}.
    \item Evaluate $a^{(1,0)}_1$ by $\mathcal{I}_{j, R}[\bar{u}_{0,R}]$ (cf.~\eqref{eq:a-I_}). Compute {\em the first-order far-field predictor} (boundary condition) by~\eqref{eq:exp_ai}
\begin{eqnarray}
\hat{g}_1 := u_0^{\rm C} + u_1^{\rm C} 
+ a^{(1,0)}_1:\nabla G_0.
\end{eqnarray}


\item  Compute {\em the first-order corrector}: $\bar{u}_{1,R} = \hat{g}_1 + \bar{w}_{1,R}$, $\bar{w}_{1,R}\in\mathcal{W}_R$ such that 
\begin{equation*}
\delta\mathcal{E}( \bar{u}_{1,R})[v_R] = 0 \qquad \forall v_R \in \mathcal{W}_R.
\end{equation*}

The desired accuracy $\lVert D\bar{u} - D\bar{u}_{1, R} \rVert_{\ell^2} \lesssim R^{-2}\log^2(R)$ and $\big|\mathcal{E}(\bar{u})-\mathcal{E}( \bar{u}_{1, R} )\big| \lesssim R^{-4} \log^{4}(R)$ is then achieved (cf.~$p=1$ in Theorem~\ref{th:galerkinaM}).
\end{enumerate}
\end{algorithm}

\begin{remark}
In principle, one may construct higher-order boundary conditions  to achieve faster convergence. For instance, when $p = 2$, Steps 2--4 of Algorithm~\ref{alg:moment_iter_a} can be naturally extended: we solve the higher-order predictor equations~\eqref{eq:uCiPDE} using the spectral Galerkin method to obtain $u_2^{\mathrm{C}}$, and compute the coefficients $a^{(1,0)}_2$ and $a^{(2,0)}_2$ via the moment estimator $\mathcal{I}_{2,R}[\bar{u}_{1,R}]$ (cf.~\eqref{eq:a-I}). Using~\eqref{eq:exp_ai}, we then construct {\it the second-order far-field predictor}:
\[
\hat{g}_2 := u_0^{\rm C} + u_1^{\rm C} + u_2^{\rm C}
+ a^{(1,0)}_2 : \nabla G_0 + a^{(2,0)}_2 : \nabla G_1.
\]
The coefficients $a^{(1,0)}_2$ and $a^{(2,0)}_2$ are iteratively refined using the moment iteration procedure described in~\cite[Algorithm 3.1]{braun2025higher} until the stopping criterion~\eqref{eq:stop} is satisfied. Finally, we solve for the second-order corrector $\bar{u}_{2,R} = \hat{g}_2 + \bar{w}_{2,R}$ with $\bar{w}_{2,R} \in \mathcal{W}_R$ such that~\eqref{eq:galerkin_approx_a} holds. This leads to improved convergence rates: $\| D\bar{u} - D\bar{u}_{2, R} \|_{\ell^2} \lesssim R^{-3} \log^3(R), ~\big| \mathcal{E}(\bar{u}) - \mathcal{E}(\bar{u}_{2, R}) \big| \lesssim R^{-6} \log^6(R)$.

However, in practice, achieving and verifying higher-order convergence is increasingly difficult, and solving the corresponding higher-order elliptic systems increases the computational cost. For these reasons, our numerical experiments focus on the case $p=1$, which already demonstrates the effectiveness of the proposed framework.
\end{remark}

\section{Numerical Experiments}
\label{sec:numexp}

In this section, we apply the main algorithm (cf.~Algorithm~\ref{alg:moment_iter_a}) to numerically validate the high-order boundary conditions established in Theorem~\ref{th:galerkinaM}. These boundary conditions enable faster convergence as the computational domain increases at a moderate computational cost.



\subsection{Model problems}
\label{sec:sub:numer_model}


In numerical experiments, we considered tungsten (W) with a body centered cubic (BCC) structure. The interatomic interactions are modeled using the Embedded Atomic Model (EAM) potential~\cite{Daw1984a}. The cutoff radius is set to $r_{\rm cut}=5.5$\AA, covering the interaction of the third neighbor in the lattice.

We consider two typical types of straight dislocations, and project the lattice to a two-dimensional lattice on the normal plane to describe their behavior. The core geometric shape of each instance is depicted on the (001) plane. As shown in Figure~\ref{figs:geom1}:
\begin{itemize}
	\item {\em Antiplane screw dislocation  (Figure~\ref{fig:sub:anti}):} $u:\Lambda\rightarrow\mathbb{R}^1$,
	\item {\em (001)[100] Edge dislocation (Figure~\ref{fig:sub:edge}):} $u:\Lambda\rightarrow\mathbb{R}^2$.
\end{itemize}


\begin{figure}[!htb]
    \centering
    \subfloat[Antiplane screw dislocation \label{fig:sub:anti}]{
    \includegraphics[height=4.0cm]{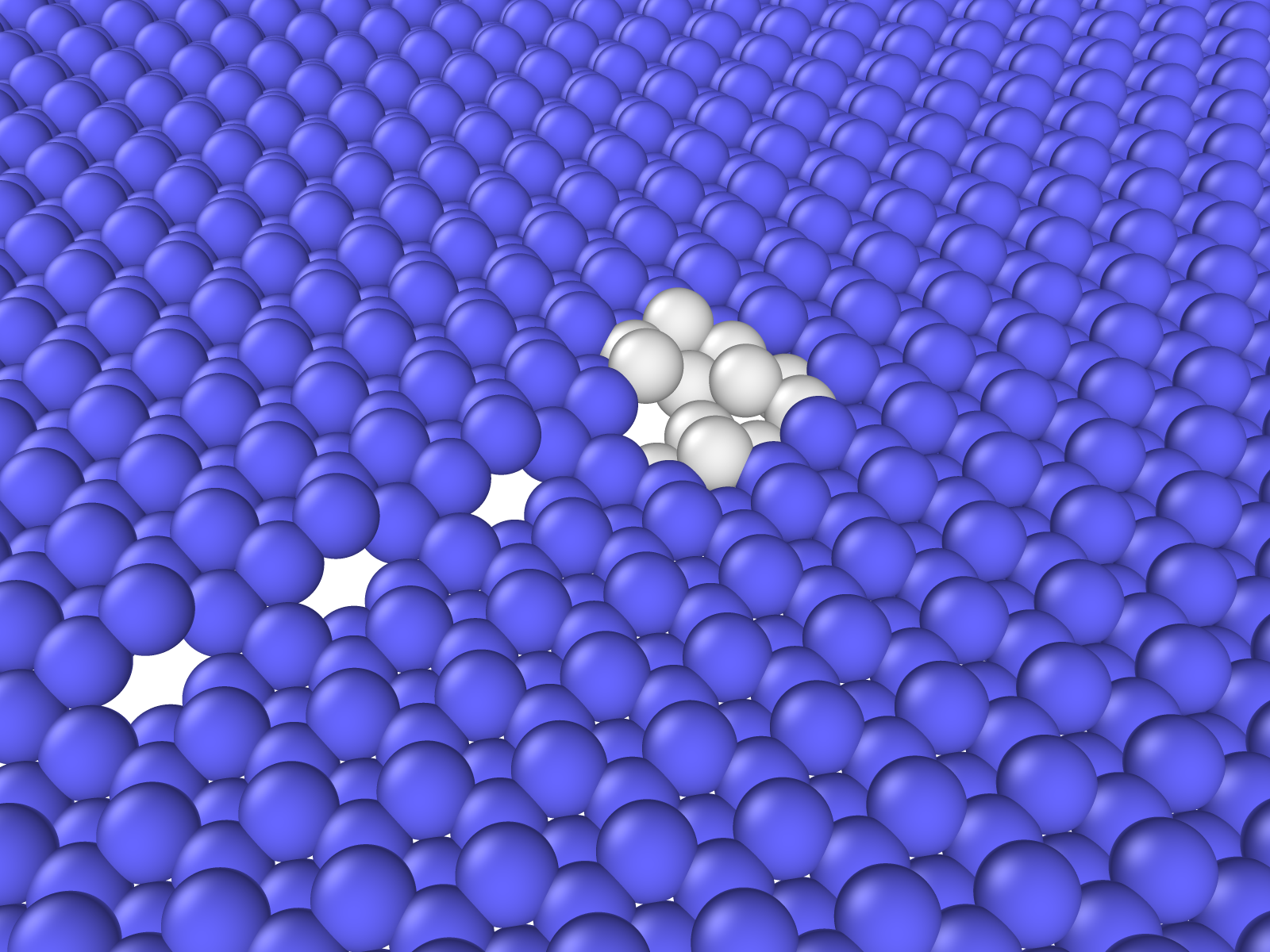}} \qquad
    \subfloat[(001)\textnormal{[100]} Edge dislocation \label{fig:sub:edge}]{
    \includegraphics[height=4.0cm]{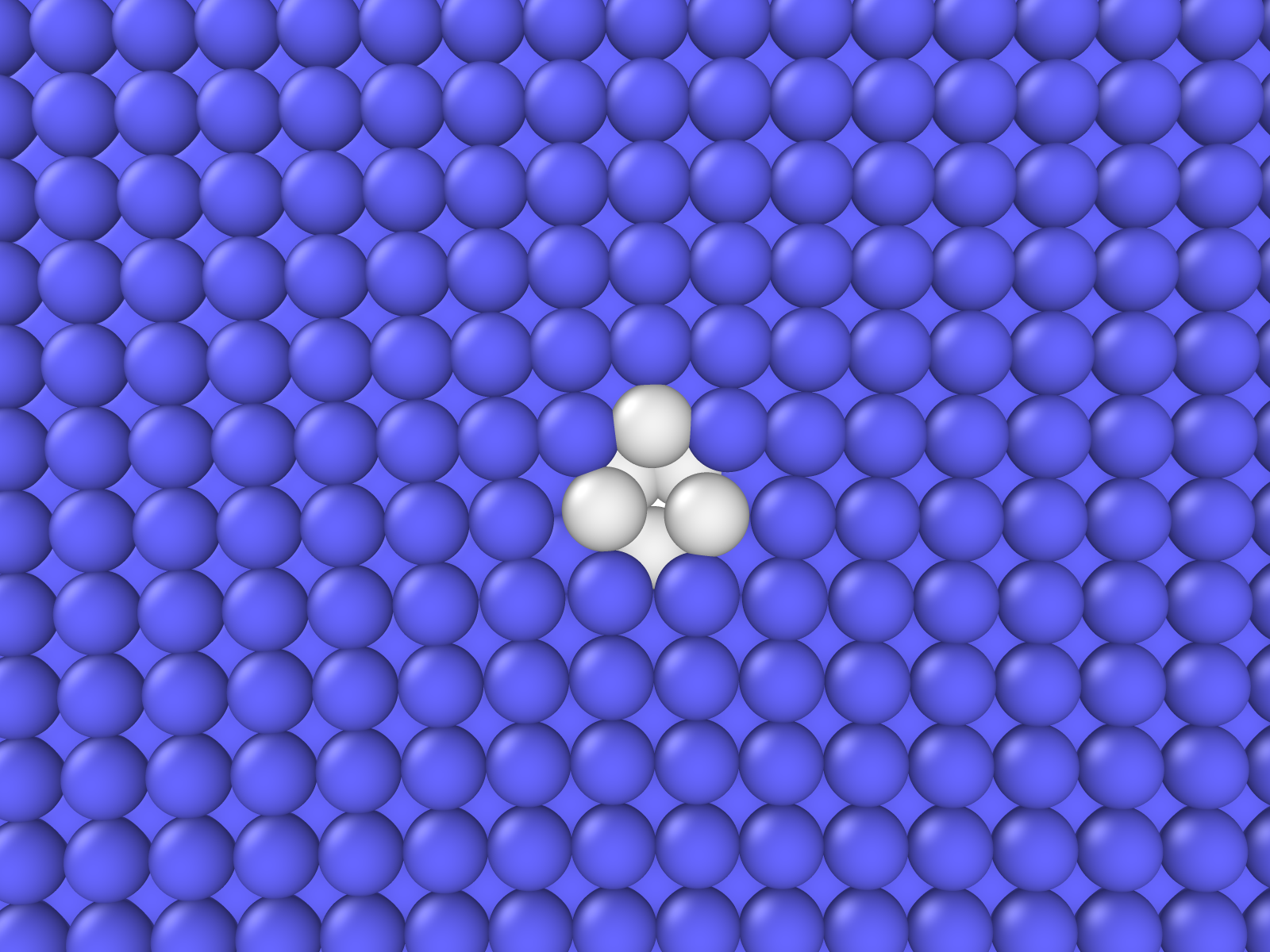}}
    \caption{Single straight dislocations in BCC Tungsten considered in this work. Colored by Common Neighbor Analysis (CNA) in Ovito~\cite{stukowski2009visualization}.} 
    \label{figs:geom1}
\end{figure}



The supercell problem~\eqref{eq:equil_edge} is solved using a preconditioned Limited-memory Broyden–Fletcher–Goldfarb–Shanno (LBFGS) algorithm~\cite{liu1989limited}, followed by a standard Newton refinement step for post-processing. The minimization terminates when the residual force satisfies \( \| \nabla \mathcal{E}(u) \|_{\infty} < \text{tol} = 10^{-8} \). In the practical implementation of the continuum predictor, the continuous Green's function \( G_0 \) is evaluated via the Barnett formula~\cite{barnett1972precise}; see also the derivation in~\cite[Sec.~6.3]{braun2025higher} for further details.

The higher-order predictor equations (cf.~\S\ref{sec:sub:pde}) are discretized and solved using the spectral Galerkin method implemented in the Python package \texttt{shenfun}~\cite{mortensen2018shenfun}. All numerical experiments are conducted on a workstation equipped with an Intel Xeon Platinum 8268 CPU (2.90\,GHz) and 361\,GB RAM.

\subsection{Convergence of cell problems}
\label{sec:sub:conv}


In this section, we evaluate the decay of strain correctors and study the convergence of the geometric error \( \|D\bar{u} - D\bar{u}_{i,R}\|_{\ell^2} \) and energy error \( |\mathcal{E}(\bar{u}) - \mathcal{E}(\bar{u}_{i,R})| \) with respect to the computational radius. The approximate equilibrium solutions \( \bar{u}_{i,R} \) are obtained via Algorithm~\ref{alg:moment_iter_a}, using zeroth- and first-order boundary conditions (\( i = 0, 1 \)). From the Galerkin solution \( \bar{u}_{p,R} \) of~\eqref{eq:galerkin_approx_a}, we extract the corrector component \( \bar{w}_{p,R} \) introduced in~\eqref{eq:u_acR}. We choose $R_{\rm c}=320$ as the computational domain size to ensure that the truncation error in Theorem~\ref{eq:galerkin_approx_a} is negligible. As a reference solution, we use the result computed on a larger domain of radius \( R_{\rm dom} = 100a_0 \), where \( a_0 \) is the lattice constant of BCC tungsten at room temperature.

\subsubsection*{Decay of strains}
\label{sec:sub:strains}


We first verify the decay behavior of \( \bar{w}_{p,R} \) in strains, as predicted by Theorem~\ref{thm:dislocation}. Figure~\ref{figs:decay-Anti-edge} illustrates the decay of strain magnitudes for different predictor orders with respect to the distance to the defect core \( |\ell| \), for both antiplane screw and edge dislocations. The transparent points correspond to the data pairs \( \big(|\ell|, |D \bar{w}_{i, R_{\rm dom}}(\ell)|\big) \) for \( i = 0, 1 \), while the solid curves indicate their envelope profiles. The improved decay rates of the higher-order predictors observed numerically are in agreement with the theoretical predictions.


\begin{figure}[!htb]
    \centering
    \includegraphics[height=4.5cm]{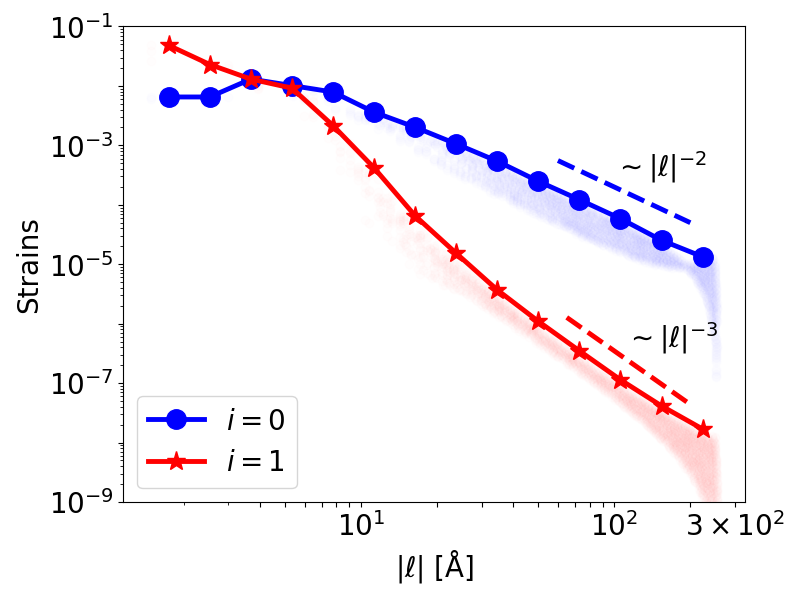}~~ 
    \includegraphics[height=4.5cm]{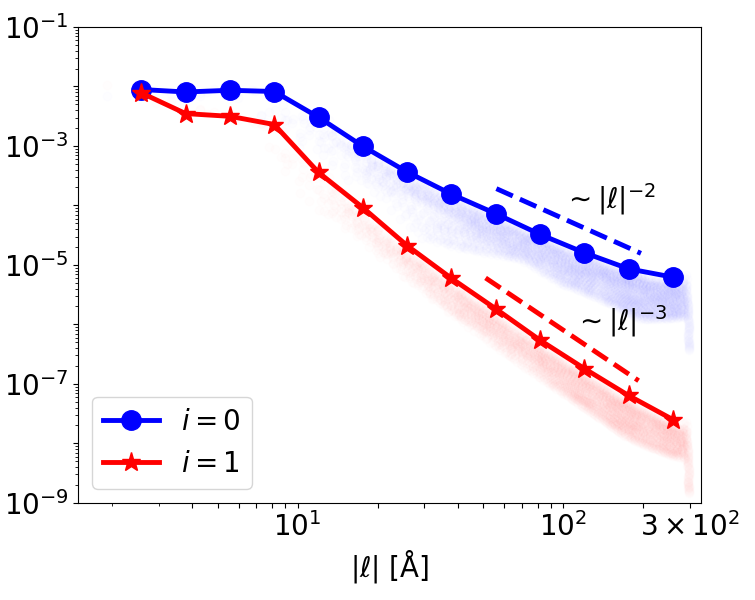}
   \caption{Decay of strains $|D \bar{w}_{i, R_{\rm dom}}(\ell)|$ as a function of the distance to the defect core $\lvert \ell \rvert$. \textbf{Left:} Antiplane screw dislocation. \textbf{Right:} (001)[100] Edge dislocation.}
    \label{figs:decay-Anti-edge}
\end{figure}
 
\subsubsection*{Geometry error}

We then consider the convergence of the geometry error \( \|D\bar{u} - D\bar{u}_{i, R}\|_{\ell^2} \) with respect to the domain size \( R \), as shown in Figure~\ref{figs:convergence-geo}. This result clearly illustrates the convergence behavior of the supercell approximation~\eqref{eq:equil_edge} for both standard and higher-order far-field predictors. In particular, the boundary conditions constructed via Algorithm~\ref{alg:moment_iter_a} lead to visibly improved convergence rates, owing to the faster decay of the {\em corrector} solutions. This observation is consistent with the theoretical prediction of Theorem~\ref{th:galerkinaM}. Notably, for \( R > 30 \), the enhanced convergence becomes more pronounced, indicating that high-order accuracy can be achieved using relatively small computational domains. This significantly reduces computational cost and enables more efficient high-accuracy electronic structure calculations for dislocations.


\begin{figure}[!htb]
    \centering
    \includegraphics[height=4.5cm]{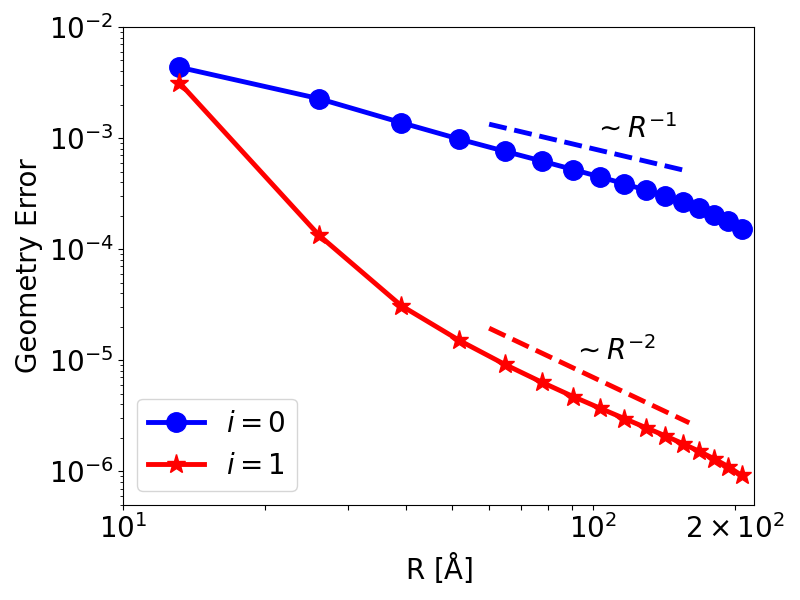}~~
    \includegraphics[height=4.5cm]{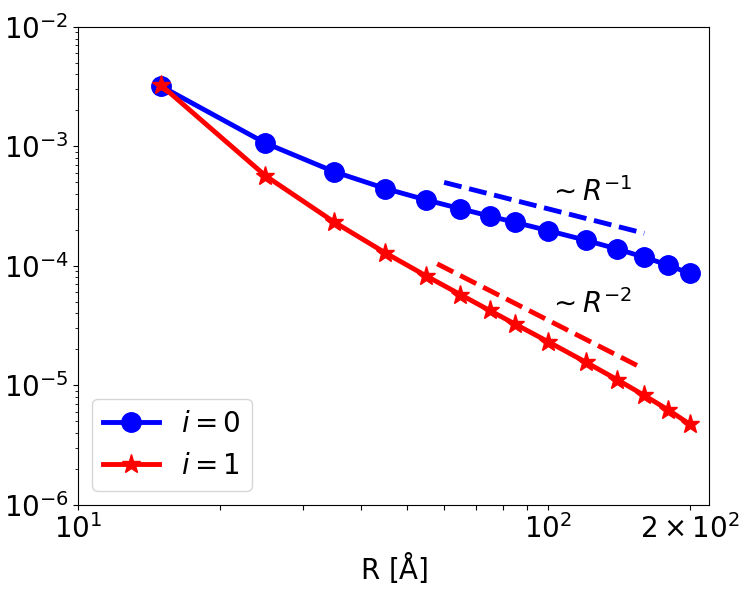}
    \caption{Convergence of geometry error $\|D\bar{u} - D\bar{u}_{i, R}\|_{\ell^2}$ for $i=0,1$ against domain size $R$. \textbf{Left:} Antiplane screw dislocation; \textbf{Right:} (001)[100] edge dislocation.} 
    \label{figs:convergence-geo}
\end{figure}
\subsubsection*{Energy error}

According to the relationship between geometric error and energy error~\eqref{eq:GE}, the convergence of the energy error is natural. In particular, we observe that the energy error \(\big|\mathcal{E}(\bar{u}) - \mathcal{E}(\bar{u}_{i, R})\big|\) for \(i=0,1\) decays with respect to the domain size \(R\), as shown in Figure~\ref{figs:convergence-energy}. The observed convergence rates match the theoretical predictions in Theorem~\ref{th:galerkinaM}.


\begin{figure}[!htb]
    \centering
    \includegraphics[height=4.5cm]{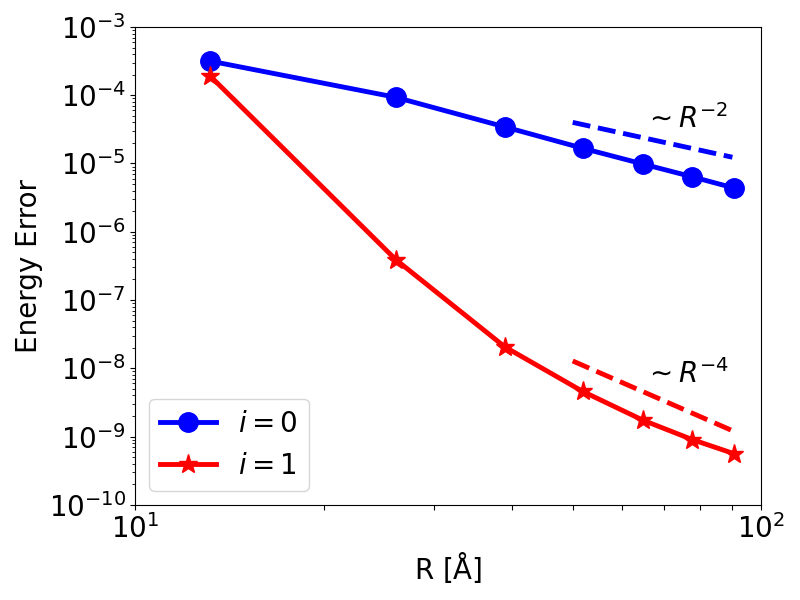}~~
    \includegraphics[height=4.5cm]{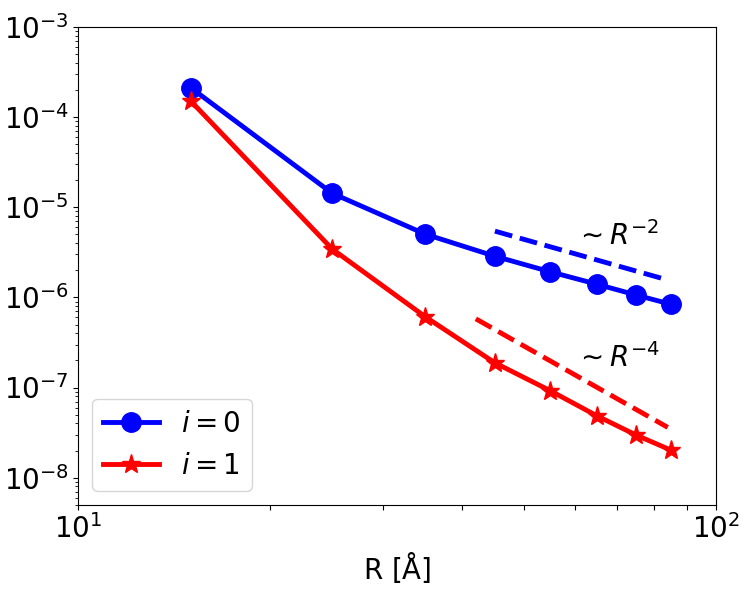}
    \caption{Convergence of energy error $\big|\mathcal{E}(\bar{u})-\mathcal{E}(\bar{u}_{i, R})\big|$ for $i=0,1$ against domain size $R$. \textbf{Left:} Antiplane screw dislocation; \textbf{Right:} (001)[100] edge dislocation.} 
    \label{figs:convergence-energy}
\end{figure}

\subsubsection*{Computational efficiency}

To evaluate the computational efficiency of Algorithm~\ref{alg:moment_iter_a}, we report the CPU time (in seconds) required to solve the two types of dislocations under different boundary conditions (\(i = 0, 1\)). 

Figure~\ref{figs:CPU-time-a} presents the total CPU time $T_{\rm tot}$ and the time devoted to constructing high-order boundary conditions $T_{\rm bc}$ for $i=1$, plotted against the number of atoms $N_{\rm at}$. The total time $T_{\rm tot}$ scales approximately linearly with $N_{\rm at}$. In addition, the boundary construction time $T_{\rm bc}$ remains significantly smaller than $T_{\rm tot}$, and the ratio $T_{\rm bc}/T_{\rm tot}$ decreases as $N_{\rm at}$ increases. This indicates that the relative cost of constructing the high-order boundary conditions remains low across all problem sizes.

Figure~\ref{figs:CPU-time-b} shows the Pareto frontiers of geometry and energy errors (represented by blue and red lines, respectively), with each point corresponding to a different $N_{\rm at}$. For a fixed $i$, both errors exhibit similar trends with respect to the total CPU time. Notably, the results for $i=1$ cluster closer to the bottom-left corner compared to those for $i=0$, indicating that high-order boundary conditions achieve comparable accuracy with substantially reduced computational cost. This clearly demonstrates the efficiency advantage of using high-order boundary conditions in practice.


\begin{figure}[!htb]
    \centering
    \subfloat[CPU time (in seconds) for $i=1$ versus the number of atoms $N_{\rm at}$.\label{figs:CPU-time-a}]{
     \includegraphics[height=4.5cm]{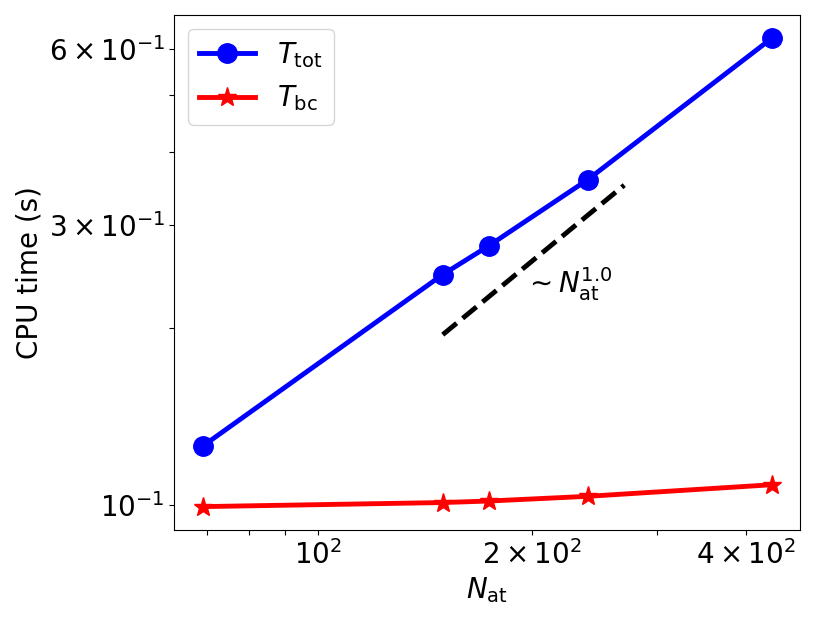}~~
    \includegraphics[height=4.5cm]{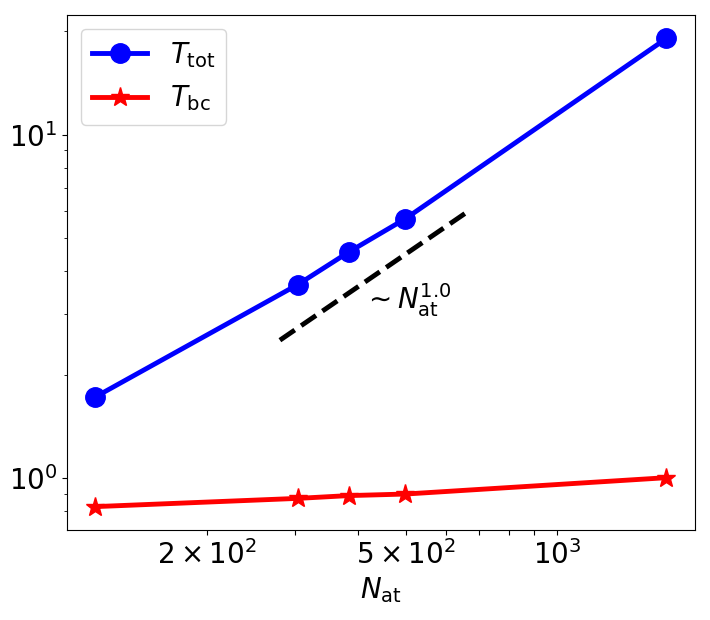}}~\\
    \hspace*{1.5mm}
    \subfloat[Geometry and energy errors for $i=0,1$ plotted against total CPU time.\label{figs:CPU-time-b}]{
    \includegraphics[height=4.5cm]{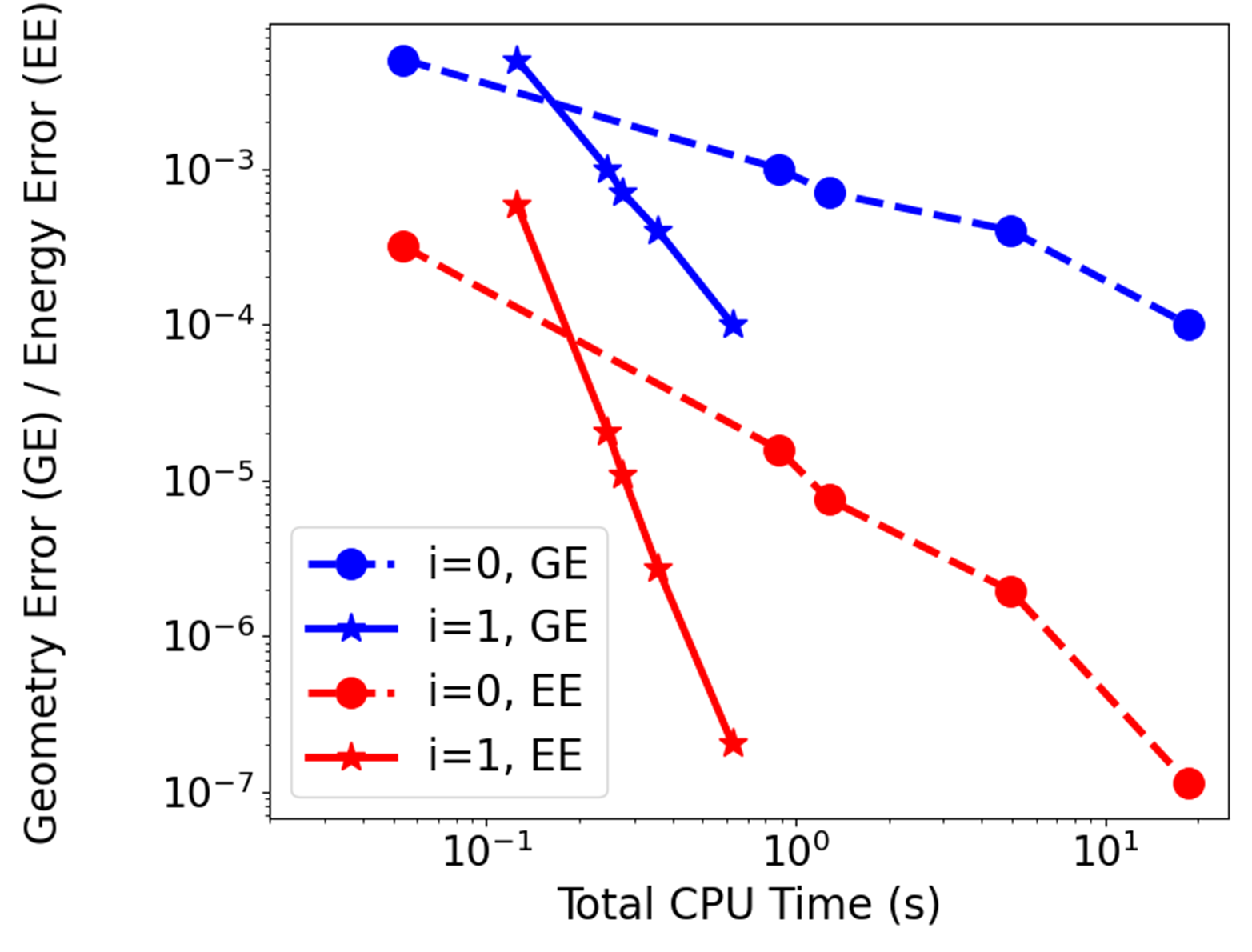}~~
\includegraphics[height=4.5cm]{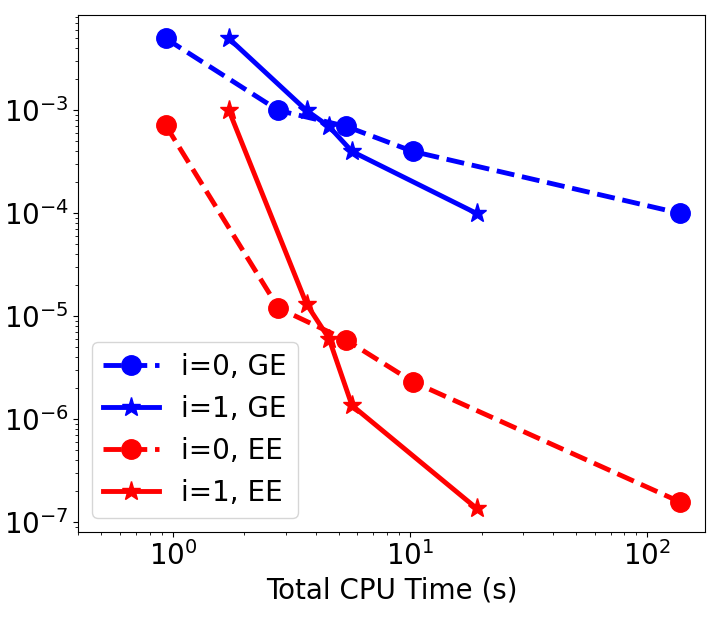}}~
    \caption{Computational efficiency of higher-order boundary conditions. \textbf{Left:} Antiplane screw dislocation; \textbf{Right:} (001)[100] edge dislocation.}
    \label{figs:CPU-time}
\end{figure}

\section{Conclusion}
\label{sec:conclusion}

In this work, we construct high-order boundary conditions for dislocations through multipole expansion of defect equilibrium. We develop a novel numerical scheme that involves employing spectral methods to solve high-order continuous PDEs, which improves the convergence of traditional cell approximations in defect simulations. To evaluate our method's effectiveness, we calculate the convergence of geometric error and energy error in numerical experiments, obtaining results consistent with theoretical predictions. Our numerical scheme is demonstrated to achieve accelerated convergence rates with respect to computational domain size.

The results discussed here focus exclusively on simple lattices and representative dislocations, establishing a foundational framework. While extending this work to multi-lattices and diverse dislocation types would require addressing additional technical challenges, there appear to be no fundamental theoretical barriers to such exploration. In numerical experiments, the equilibrium state of defects can be expanded to higher-order terms, enabling the implementation of higher-order boundary conditions. However, this requires solving high-order PDEs, which introduces further technical complexities. These challenges will be addressed in future work.


\appendix
\section{Preliminaries Results in \S~\ref{sec:disloc}}
\label{app:premilinary results}

In this section, we present the premilinary results as a supplement to the main text. Before that, let us introduce some more notations.


The symmetric tensor product for $\boldsymbol{\sigma}$, yielding a vector subspace $(\mathbb{R}^d)^{\odot k}$ within $(\mathbb{R}^d)^{\otimes k}$, is defined as $\boldsymbol{\sigma}^\odot:=  \sigma^{(1)} \odot \cdots \odot \sigma^{(k)} := {\rm sym~} \boldsymbol{\sigma}^\otimes := \frac{1}{k!} \sum_{g \in S_k} g(\boldsymbol{\sigma})^\otimes$, where $S_k$ is the symmetric group encompasses all permutations of $\{1, \ldots, k\}$, and $g(\sigma):=(\sigma^{g(1)},...,\sigma^{g(k)})$ for any $g\in S_k$ and $\sigma\in (\mathbb{R}^d)^{k}$. 
For the scalar product on these spaces, we denote it by $A : B$ for $A, B \in (\mathbb{R}^d)^{\otimes k}$, defined as the linear extension of $\boldsymbol{\sigma}^\otimes : \boldsymbol{\rho}^\otimes = \prod_{i=1}^k \sigma^{(i)} \cdot \rho^{(i)}$. 

\subsection{The properties of the site energy potential \texorpdfstring{$V$}{V}}
\label{app:V}
We summarize the assumptions and properties of the site energy potential \( V \) used throughout the analysis. To this end, sufficiently large radii $\hat{r}_{\mathscr{A}}$, and $\hat{m}_{\mathscr{A}}$ are selected to define the space $\mathscr{A}$ to restrict the admissible corrector displacements.
\begin{displaymath}
  \label{eq:disl:defn_Adm}
  \mathscr{A} := \big\{ u : \Lambda \to \mathbb{R}^N \mid \| \nabla u \|_{L^\infty} < \hat{m}_{\mathscr{A}}
  \text{ and } |\nabla u(x)| < 1/2 \text{ for } |x| > \hat{r}_{\mathscr{A}} \big\}.
\end{displaymath}
With $\hat{r}_{\mathscr{A}}$, and $\hat{m}_{\mathscr{A}}$ chosen sufficiently large, any equilibrium solution is guaranteed to lie within $\mathscr{A}$.
For antiplane screw dislocations, $\mathscr{A}$ may be taken as $\dot{\mathscr{W}}^{1,2}$, as only slip-invariance in the antiplane direction is required, that is, the topology of the projected 2D lattice remains unchanged.

Next, the slip operator $S_0$ acting on a displacement $w : \Lambda \to \mathbb{R}^N$ 
is defined by
\begin{equation}
\label{eq:disl:defn_Sop}
S_0 w(x) := 
\left\{
\begin{aligned}
&w(x), & x_2 > \hat{x}_2, \\
&w(x - \sf b_{12}) - \sf{b}, & x_2 < \hat{x}_2.
\end{aligned}
\right.
\end{equation}
For $w = u_0^{\rm C} + u, u \in \mathscr{A}$, we shall write $S_0 w = S_0 u_0^{\rm C} + Su$, where $S$
is an $\ell^2$-orthogonal operator, with dual $R = S^* =
S^{-1}$, \label{sym:S-R}
\begin{align*}
Su(\ell) 
:= \left\{ 
\begin{aligned}
&u(\ell), & \ell_2 > \hat{x}_2, \\
&u(\ell-\sf b_{12}), & \ell_2 < \hat{x}_2, 
\end{aligned}
\right.
\quad 
\text{and} \quad
Ru(\ell) 
:= \left\{
\begin{aligned}
&u(\ell), & \ell_2 > \hat{x}_2, \\
&u(\ell+\sf b_{12}), & \ell_2 < \hat{x}_2.
\end{aligned}
\right.
\end{align*}

We assume that $V$ is invariant under lattice slip:
\begin{equation}
\label{eq:disl:slip_invariance}
V\big(D(u_0^{\rm C}+u)(\ell)\big) 
= V\big(R D S_0(u_0^{\rm C}+u)(\ell)\big)
\qquad \forall u \in \mathscr{A}, \ell \in \Lambda.
\end{equation}

\subsection{The relationship between force moments and discrete and continuous coefficients}
\label{app:sub:a}
We establish the explicit relationship between force
moments and coefficients in continuous expansion.

For $i \in \mathbb{N}$, if $\ell \mapsto H[u](\ell) \otimes \ell^{\otimes i} \in \ell^1(\Lambda)$, we introduce the $i$-th force moment
\begin{equation}
\mathcal{I}_i[u] = \sum_{\ell \in \Lambda} H[u](\ell) \otimes \ell^{\otimes i},
\label{eq:def_Ii}
\end{equation}
which provides the essential mathematical structure for our subsequent analysis.

The discrete coefficients $b^{(i,k)}$ are theoretically computable via a linear transformation, as established in~\cite[Lemma 5.6]{braun2022asymptotic}. This relationship is explicitly given by:
\begin{equation}
\label{eq:bIrelation}
\big(\mathcal{I}_i(\bar{u})\big)_{\cdot k} = (-1)^i i! \sum\limits_{\rho \in \mathcal{S}^i} (b_{\text{exact}}^{(i,k)})_\rho\cdot\rho^{\odot},
\end{equation}
with the force moments $\mathcal{I}_i$ given in~\eqref{eq:def_Ii}.

For $p=1,2$, the framework in \S~\ref{sec:sub:coeffts} only needs the continuous Green's function $G_0$ to solve the atomic continuum error. On this basis, we can establish a completely continuous multipole expansion.
\begin{align}\label{eq:exp_a}
\bar{u} = \sum_{i=0}^{p}u_i^{\rm C} + \sum_{k=1}^N \Big(\sum_{i=1}^{p} a^{(i,0,k)}:\nabla^i (G_0)_{\cdot k} \Big)+ w 
\end{align}


By combining \eqref{eq:exp_b} with \eqref{eq:exp_a}, we derive a precise relationship between the coefficients $a^{(i,n,k)}$ and $b^{(i,k)}$ through Taylor expansion of the discrete difference stencil. For $1\leq i\leq 2$, $1\leq k \leq N$, and $n=0$, this relationship takes the form:
\begin{equation}
\begin{aligned}
\label{eq:a-b}
(a^{(1,0,k)})_{\cdot j}  &= \sum_{\rho \in \mathcal{R}} (b^{(1,k)})_\rho \cdot \rho_j,\\
(a^{(2,0,k)})_{\cdot jm} &= \sum_{\rho, \sigma \in \mathcal{R}} (b^{(2,k)})_{\rho\sigma} \cdot \rho_j \sigma_m + \frac{1}{2} \sum_{\rho\in\mathcal{R}} (b^{(1,k)})_\rho \cdot \rho_j\rho_m.
\end{aligned}
\end{equation}
Furthermore, incorporating~\eqref{eq:bIrelation} yields a direct connection between the continuous coefficients and moments. Specifically, for each $i$ and $n$, when $a^{(i,n,\cdot)}$ denotes the collection of $(a^{(i,n,k)})_{\cdot j}$ for all $j,k$, we obtain:
\begin{equation}\label{eq:a-I}
a^{(1,0,\cdot)} = -\mathcal{I}_1[\bar{u}], \quad a^{(2,0,\cdot)} = \frac{1}{2}\mathcal{I}_2[\bar{u}], 
\end{equation}
This provides a computational framework for determining the continuous coefficients $a^{(i,n,\cdot)}$ via force moments. The detailed derivations for this computational process are available in~\cite{braun2025higher}.

\subsection{Proof of Theorem~\ref{thm:dislocation}}
\label{app:proof-thm-disloc}
We now present the proof of Theorem~\ref{thm:dislocation}. Our argument builds upon the approach developed in~\cite[Theorem 7]{braun2022asymptotic} for screw dislocations and extends the analysis to accommodate edge dislocations as well.
\begin{proof}[Proof of Theorem \ref{thm:dislocation}]
\noindent\textit{The Case ($p=0$)}: 
From the assumptions of Theorem~\ref{thm:dislocation}, we have $J \leq K-2$. The continuum displacement field $u_0^{\rm C}$ is defined according to the dislocation model presented in \S~\ref{sec:disloc} as:
\begin{equation}
\label{eq:u_0C}
u_0^{\rm C} := \begin{cases}
u^{\rm lin} & \text{for the antiplane screw dislocations,}\\
u^{\rm lin}\circ\xi^{-1} & \text{for the edge dislocations}.
\end{cases}
\end{equation}
The result from~\cite{braun2022asymptotic} demonstrates that for antiplane screw dislocations $u_0^{\rm C} = u^{\rm lin}$, with the error estimate
\begin{equation}\label{eq:antierrorlowestorder}
\big\lvert D^j (\bar{u}-u_0^{\rm lin}) (\ell) \big\rvert \lesssim \lvert{\ell}\rvert^{1-d-j} \log (\lvert{\ell}\rvert),\quad 1\leq j \leq J.
\end{equation}
while $\nabla u^{\rm lin} \in C^\infty(\mathbb{R}^d\backslash\{0\};\mathbb{R}^N)$ with $\lvert \nabla^j u^{\rm lin} (\ell) \rvert \lesssim \lvert{\ell}\rvert^{-j},~\text{for all}\ j\geq 1.$

We extend these results to edge dislocations and $u_0^{\rm C} = u^{\rm lin}\circ\xi^{-1}$ in this case. Following~\cite[Lemma 3.1(iii)]{EOS2016}, the relationship between $\nabla^ju^{\rm C}_0$ and $\nabla^ju^{\rm lin}$ is given by
\begin{equation*}
\nabla^ju_0^{\rm C} = \nabla^ju^{\rm lin} + O(\lvert\ell\rvert^{-1-j}).
\end{equation*}
This leads to the error estimate:
\begin{equation}\label{eq:roneestimate}
\begin{aligned}
\left\lvert D^j (\bar{u}-u_0^{\rm C}) (\ell) \right\rvert =&\big\lvert [D^j\bar{u}-D^ju^{\rm lin}(\xi^{-1}(\ell))] (\ell) \big\rvert\\
\leq&\big\lvert D^j\bar{u}(\ell)-D^ju^{\rm lin}(\ell) \rvert+ O(\lvert \ell\rvert^{-1-j})\\
\lesssim& \lvert{\ell}\rvert^{-1-j}\log (\lvert{\ell}\rvert),\quad 1\leq j \leq J.
\end{aligned}
\end{equation}
When $\nabla u_0^{\rm C} \in C^\infty(\mathbb{R}^d\backslash\{0\};\mathbb{R}^N)$, for $j\geq 1$, $\nabla u_0^{\rm C}$ satisfies
\begin{equation}
\label{eq:uzeroestimate-edge}
\big\lvert \nabla^j u_0^{\rm C} (\ell) \big\rvert = \big\lvert \nabla^j u^{\rm lin}(\xi^{-1}(\ell)) (\ell) \big\rvert = \big\lvert \nabla^j u^{\rm lin}(\ell)\big\rvert + O(\lvert\ell\rvert^{-1-j}) \lesssim \lvert{\ell}\rvert^{-j}.
\end{equation}

Hence, we have for antiplane screw and edge dislocations,
\begin{equation}\label{eq:Dislerrorlowestorder}
  \big\lvert D^j (\bar{u}-u_0^{\rm C}) (\ell) \big\rvert \lesssim \lvert{\ell}\rvert^{1-d-j} \log (\lvert{\ell}\rvert),\quad 1\leq j \leq J.
\end{equation}
while $\nabla u_0^{\rm C} \in C^\infty(\mathbb{R}^d\backslash\{0\};\mathbb{R}^N)$ with $\lvert \nabla^j u_0^{\rm C} (\ell) \rvert \lesssim \lvert{\ell}\rvert^{-j},~\text{for all}\ j\geq 1$.


\noindent\textit{The Case ($p \geq 1$)}: The proof follows the same approach as in~\cite[Section 7.2]{braun2022asymptotic}, since the distinction between edge and screw dislocations primarily affects the definition of the leading-order term $u_0^{\rm C}$ by~\eqref{eq:u_0C}. For $p=0$, we have already established that edge dislocations yield the same error estimates~\eqref{eq:Dislerrorlowestorder} as those derived for screw dislocations in~\cite[Section 7.1]{braun2022asymptotic}. Hence, the difference in dislocation type does not influence the subsequent analysis for $p \geq 1$, and the same proof strategy applies.

\end{proof}

\section{Supplementary Results in \S~\ref{sec:Scheme} and \S~\ref{sec:pde}}
\label{app:supply and proof}
\subsection{Analysis of multipole moment approximations}
\label{app:sec:b}

Let $b$ and $b_{\rm exact}$ denote the complete collections of tensors $b^{(i,k)}$ and $b^{(i,k)}_{\rm exact}$ respectively, for all possible $(i,k)$ pairs. For a given $b$, solutions in $\mathcal{U}_{p,R}^{(b)}$ yield a specific case of Theorem~\ref{th:galerkin}.
\begin{lemma}\label{th:galerkinfixedb}
Suppose that $\bar{u}$ is a strongly stable solution of \eqref{eq:equil_edge}. Denote the numerical error ${\varepsilon}^{\rm num}:=R_{\rm c}^{-1}\log(R_{\rm c}) +e^{-cN_{\rm pde}}$. For $R$ sufficiently large, there exists a solution $\bar{v}_{p,R}^{\rm b}\in \mathcal{U}_{p,R}^{\rm (b)}$, $\bar{u}_{p,R}^{\rm b}=\hat{u}^{\rm num}_{p} + \bar{v}_{p,R}^{\rm b}$  to \eqref{eq:galerkin_approx} such that
\begin{equation}\label{eq:ubR}
 \big\| D\bar{u} - D\bar{u}_{p,R}^{\rm b} \big\|_{\ell^2}
 \lesssim R^{-1 - p} \log^{p+1}(R) + \sum_{i=1}^p \big\lvert b^{(i,\cdot)} - b_{\rm exact}^{(i,\cdot)} \big\rvert R^{-i}+ {\varepsilon}^{\rm num}.
\end{equation}
\end{lemma}
\begin{proof}[Proof of Lemma~\ref{th:galerkinfixedb}]
The target geometric error $\big\| D\bar{u} - D\bar{u}_{p,R}^{\rm b} \big\|_{\ell^2}$ is split as:
\begin{equation}
\label{eq:DubR_split}
\begin{aligned}
\big\| D\bar{u} - D\bar{u}_{p,R}^{\rm b} \big\|_{\ell^2} &\leq \big\| D\bar{u} - D\bar{u}_{p,R}^{\rm c} \big\|_{\ell^2} + \big\| D\bar{u}_{p,R}^{\rm c} - D\bar{u}_{p,R}^{\rm b} \big\|_{\ell^2}\\
&=\big\| D\bar{u} - D\bar{u}_{p,R}^{\rm c} \big\|_{\ell^2} + \big\| Dv^{*}_{p,R} - D\bar{v}_{p,R}^{\rm b} \big\|_{\ell^2}.
\end{aligned}
\end{equation}
The first term is bounded by Theorem~\ref{th:galerkin} and \eqref{eq:approx-pde} as
\begin{equation}
\begin{aligned}
\big\| D\bar{u} - D\bar{u}_{p,R}^{\rm c}\big\|_{\ell^2} &\leq \big\| D\bar{u} - Du^{*}_{p,R} \big\|_{\ell^2} + \big\| Du^{*}_{p,R} - D\bar{u}_{p,R}^{\rm c} \big\|_{\ell^2}\\
&\lesssim R^{- 1 - p} \log^{p+1}(R) + R_{\rm c}^{-1}\log(R_{\rm c}) +e^{-cN_{\rm pde}}.
 \end{aligned}
\end{equation}

The second term in~\eqref{eq:DubR_split} corresponds to the numerical error in the moment tensor approximation, note that the respective solutions reside in the two spaces defined by \eqref{eq:exact_space_CLE} and \eqref{eq:u_bcR}. According to \cite[Lemma 6.2]{EOS2016}, the lattice Green's function satisfies $\big|D^i \mathcal{G}(\ell)\big|\leq C(1+|\ell|)^{-d-i+2}$. Therefore,
\begin{equation}
\label{eq:error-b}
\begin{aligned}
\big\| Dv^{*}_{p,R} - D\bar{v}_{p,R}^{\rm b} \big\|_{\ell^2} 
&\lesssim \sum_{i=1}^p \big\lvert b^{(i,\cdot)} - b_{\rm exact}^{(i,\cdot)} \big\rvert \cdot \Big(\sum_{|\ell|>R}\big|D^i\mathcal{G}(\ell)\big|^2\Big)^{1/2}\\ &\lesssim \sum_{i=1}^p \big\lvert b^{(i,\cdot)} - b_{\rm exact}^{(i,\cdot)} \big\rvert \cdot R^{-i}.
\end{aligned}
\end{equation}
Combining the above estimates via \eqref{eq:DubR_split} completes the proof.
\end{proof}

\subsubsection*{Moment iteration}
\label{subsec:moment iteration}

Leveraging the linear transformation~\eqref{eq:bIrelation}, our approach shifts the focus towards evaluating the force moments $\mathcal{I}_i$ instead of $b^{(i, \cdot)}$.


We first introduce a truncation operator following the constructions in~\cite{2014-dislift}. Let $\eta_R: \Lambda\rightarrow\mathbb{R}$ be a smooth cut-off function such that $\eta_R=1$ for $\lvert \ell \rvert \leq R/3$, $\eta_R =0$ for $\lvert \ell \rvert > 2R/3$ and $\lvert \nabla^j \eta_R \rvert \leq C_j R^{-j}$ for $0\leq j\leq3$. Then, we define the truncated force moments \eqref{eq:def_Ii} by 
\begin{equation} \label{eq:defn_IjR}
  \mathcal{I}_{i,R}[u] := \sum_{\ell \in \Lambda} \big(H [u](\ell) \otimes \ell^{\otimes i}\big)\cdot \eta_R(\ell).
\end{equation}

The {\em stopping criterion} for the moment iteration is given by
\begin{equation}\label{eq:stop}
\big\lvert b_M^{(i,\cdot)} - b_{\rm exact}^{(i,\cdot)} \big\rvert = O\big(R^{i-1-p} \log^{p+1}(R)\big), \qquad \text{for all } 1 \leq i \leq p.
\end{equation}
This ensures that the estimate in~\eqref{eq:ubR} achieves the optimal convergence rate of $R^{-1-p} \log^{p+1}(R)$, up to a numerical error $\varepsilon^{\rm num}$, as stated in~\eqref{eq:approx-b-R}.

\subsection{Proof of Theorem \ref{th:galerkinaM} }
\label{app:proof-of-thm3.2}

Our goal in this section is to analyze the impact of the numerical approximations we introduced in \S~\ref{sec:sub:approximation} on the convergence rates of the cell problems.
\begin{proof}[Proof of Theorem \ref{th:galerkinaM}]
We start by decomposing the geometry error into four components:
\begin{equation}
\label{eq:error-split}
\begin{aligned}
\big\| D\bar{u} - D\bar{u}_{p,R} \big\|_{\ell^2} \leq& \big\| D\bar{u} - Du^{*}_{p,R} \big\|_{\ell^2}  +\big\| Du^{*}_{p,R} - D\bar{u}_{p,R}^{\rm c} \big\|_{\ell^2} \\
&+\big\| D\bar{u}_{p,R}^{\rm c} - D\bar{u}_{p,R}^{\rm b} \big\|_{\ell^2} +\big\| D\bar{u}_{p,R}^{\rm b} - D\bar{u}_{p,R} \big\|_{\ell^2}.
\end{aligned}
\end{equation}
The first part can be directly bounded by the estimates
in Theorem~\ref{th:galerkin}:
\begin{equation}
\big\| D\bar{u} - Du^{*}_{p,R} \big\|_{\ell^2} \leq C_{\rm G} \cdot 
 R^{- 1- p} \cdot \log^{p+1}(R).
\end{equation}
For the fourth part in~\eqref{eq:error-split}, which corresponds to the approximate error of the continuous coefficients of multipole expansion. According to the two sapces defined in~\eqref{eq:u_bcR} and~\eqref{eq:u_acR}, the error estimate of the fourth term in~\eqref{eq:error-split} is
\begin{equation}
\label{eq:error-a}
\big\| D\bar{u}_{p,R}^{\rm b} - D\bar{u}_{p,R} \big\|_{\ell^2} 
\lesssim R^{-1-p},
\end{equation}
where it follows from the estimates in~\cite[Lemma 18]{braun2022asymptotic}.

Now combining the four estimates~\eqref{eq:thm3.1-1},~\eqref{eq:approx-pde},~\eqref{eq:approx-b-R} and~\eqref{eq:error-a} with the decomposition \eqref{eq:error-split} gives the geometry error. The energy error can be directly estimated by applying
\begin{equation}
\label{eq:GE}
\big|\mathcal{E}(\bar{u})-\mathcal{E}(\bar{u}_{i, R})\big| \lesssim \big\lVert D\bar{u} - D\bar{u}_{i, R} \big\rVert^2_{\ell^2} \qquad \textrm{for}~0\leq i\leq 2. 
\end{equation}
This yields the stated results.
\end{proof}

\subsection{Proof of Lemma~\ref{thm:pde}}
\label{app:proof-thm-pde}

We are ready to present the proof of the error estimates arising from the numerical approximation of the higher-order CLE equations.



\begin{proof}[Proof of Lemma~\ref{thm:pde}] 

Let $u_{i,R_{\rm c}}^{\rm C}$ be the solution to~\eqref{u1-pde}. From~\eqref{eq:rescale}, we have $u_{i,R_{\rm c}}^{\rm C} = r\cdot v_{i,R_{\rm c}}$ and  \(u_{i,N_{\rm pde}}^{\rm C} = r\cdot v_{i,N_{\rm pde}}\).
We begin by decomposing the numerical error into two parts:
\begin{equation}
\label{eq:error-pde-split}
\begin{aligned}
\big\| \nabla u_{i}^{C} - \nabla u_{i,N_{\rm pde}}^{\rm C}\big\|_{L^2} &
\lesssim \big\| \nabla u_{i}^{C} - \nabla u_{i,R_{\rm c}}^{\rm C} \big\|_{L^2} + \big\| \nabla v_{i,R_{\rm c}} - \nabla v_{i,N_{\rm pde}}\big\|_{L^2}.
\end{aligned}
\end{equation}
We first give the estimate of the first part, it is the truncation error of the domain. Denote $e_i =u_{i}^{C} - u_{i,R_{\rm c}}^{\rm C}$, which satisfies
\begin{equation}
\label{eq:error-Rc}
\begin{aligned}
\begin{cases}
-\divo(\mathbb{C} \nabla e_i) = 0 \quad &\text{in } B_{R_{\rm c}},\\
e_i= u_i^{\rm C} \quad &\text{on } \partial B_{R_{\rm c}}.
\end{cases}
\end{aligned}
\end{equation}

To separate scale effects, we introduce a change of variables by scaling. The scaled coordinate $\bm{r}_{B_1} \in B_1$ is related to the original coordinate $\bm{r}_{\rm c} \in B_{R_{\rm c}}$ by the relation $\bm{r}_{\rm c} = R_{\rm c} \bm{r}_{B_1}$. This mapping projects points from the disk $B_{R_{\rm c}}$ onto the unit disk $B_1$.
The choice of scaling factors for the functions is motivated by the decay estimates of the original solutions. Recall from~\eqref{eq:uCestimate} and the formulations of $\mathcal{S}_i$ that for sufficiently large $\lvert\ell\rvert$, \(\big\lvert \nabla u_i^{\rm C} (\ell)\big\rvert \leq C \lvert \ell \rvert^{-i-1} \log^i (\lvert \ell \rvert)\) and \(\lvert \nabla \mathcal{S}_i\rvert\lesssim \lvert \ell\rvert_{0}^{-3-i}\log^{i-1}(\lvert \ell \rvert).\)

To ensure that the corresponding scaled functions $\hat{u}_i$ and $\hat{f}_i$ remain uniformly bounded on the unit disk $B_1$ independently of $R_c$, we introduce scaling factors $R_c^i$ and $R_c^{i+1}$, respectively. Specifically, we define:
\(\hat{u}_i(\bm{r}_{B_1}) := R_{\rm c}^i \cdot
u_{i}^{\rm C}(R_{\rm c} \bm{r}_{B_1})\), \(\hat{u}_{i,R_{\rm c}}(\bm{r}_{B_1}) := R_{\rm c}^i \cdot u_{i,R_{\rm c}}^{\rm C}(R_{\rm c} \bm{r}_{B_1})\), \(\hat{e}_i(\bm{r}_{B_1}) :=R_{\rm c}^i \cdot e_i(R_{\rm c} \bm{r}_{B_1})\), \(\hat{f}_i(\bm{r}_{B_1}) := R_{\rm c}^{i+1} \cdot
 f_i(R_{\rm c} \bm{r}_{B_1})\). 
It is straightforward to verify that $\hat{u}_i\in L^2(B_1)$ and $\hat{f}_i\in L^2(B_1)$. Under this scaling, \eqref{eq:uCiPDE} and~\eqref{eq:error-Rc} become the following scaled equations on $B_1$, respectively:
\begin{equation}
\begin{aligned}
& -\divo(\mathbb{C} \nabla \hat{u}_i) = \hat{f}_i, \quad \text{in } B_1
\hspace{1em} {\rm and} \hspace{1em} 
\left\{
\begin{array}{ll}
-\divo(\mathbb{C} \nabla \hat{e}_i) = 0, & \text{in } B_1, \\
\hat{e}_i = \hat{u}_i, & \text{on } \partial B_1
\end{array}
\right..
\end{aligned}
\label{eq:compact_pair}
\end{equation}



Suppose $\mathbb{C}$ is positive definite, from the elliptic regularity theory~\cite{evans2022partial}, the solution of \eqref{eq:compact_pair} satisfies the priori estimate: there exists constant 
$\beta_1 > 0$ (depending only on $B_1$) such that:
\begin{equation*}
\label{eq:regularity}
\|\hat{u}_i\|_{H^2(B_1)} \leq \beta_1 \left( \|{\hat{f}_i}\|_{L^2(B_1)} + \|{\hat{u}_i}\|_{L^2(B_1)} \right)\lesssim \beta_1.
\end{equation*}
By the trace theorem~\cite{evans2022partial}, there exists a bounded linear operator 
$T: H^2(B_1) \to H^{3/2}(\partial B_1)$, then there exists constant $\beta_2 > 0$
 such that:
\begin{equation*}
\|{\hat{u}_i}\|_{H^{3/2}(\partial B_1)} = \|{T\hat{u}_i}\|_{H^{3/2}(\partial B_1)} \leq \beta_2 \|{\hat{u}_i}\|_{H^2(B_1)}\lesssim \beta_1\beta_2. 
\end{equation*}
Next we give the the priori estimate for the solution of the homogeneous elliptic equation~\eqref{eq:compact_pair}: there exists constant $\beta_3 > 0$ (depending only on $B_1$) such that:
\begin{equation*}
\|{\hat{e}_i}\|_{H^1(B_1)} \leq \beta_3 \|{\hat{u}_i}\|_{H^{1/2}(\partial B_1)}\lesssim \beta_3\|{\hat{u}_i}\|_{H^{3/2}(\partial B_1)} \lesssim \beta_1\beta_2\beta_3=:\hat{\beta}_i
\end{equation*}
From the definition of the $H^1$ norm we know \(\|{\nabla \hat{e}_i}\|_{L^2(B_1)} 
\leq \|{\hat{e}_i}\|_{H^1(B_1)} \leq \beta_i\), 
where $\hat{\beta}_{i}$ is independent of $R_{\rm c}$. Now transform the estimate on the unit disk back to the original domain $B_{R_{\rm c}}$. 
Compute the $L^2$ norm of the original error gradient:
\begin{equation*}
\|{\nabla e_i}\|_{L^2(B_{R_{\rm c}})}^2 = \int_{B_{R_{\rm c}}} \lvert{\nabla e_i}\rvert^2 \, \mathrm{d}\bm{r}_{\rm c} = 
\int_{B_1} \Big\lvert{\frac{1}{R_{\rm c}^{i+1}} \nabla \hat{e}_i}\Big\rvert^2 R_{\rm c}^2 \, \mathrm{d}\bm{r}_{B_1}= 
\frac{1}{R_{\rm c}^{2i}} \|{\nabla \hat
{e}_i}\|_{L^2(B_1)}^2 \leq \frac
{\hat{\beta}_{i}}{R_{\rm c}^{2i}} 
\end{equation*}
Then we get the estimate of the truncation error
\begin{equation}
\label{eq:truncation error}
\big\|\nabla u_i^{\rm C} - \nabla u_{i,R_{\rm c}}^{\rm C}\big\|_{L^2(B_{R_{\rm c}})}\leq \hat{\beta}_i R_{\rm c}^{-i}
\end{equation}
where the constant $\hat{\beta}_i$ is independent of the truncation radius $R_{\rm c}$.

The second term in~\eqref{eq:error-pde-split} arises from the numerical error associated with solving~\eqref{eq:coordinate transformation} using the spectral Legendre-Galerkin method. According to Lemma~\ref{lem:analytic_solution}, under the condition that $v_i$ is analytic the method can achieve the following spectral accuracy~\cite{2011Spectral}:
\begin{equation}
\label{eq:error-spectral}
\| \nabla v_{i,R_{\rm c}} - \nabla v_{i,N_{\rm pde}}\|_{L^2}\leq C_{\rm s}\,e^{-cN_{\rm pde}}.
\end{equation}
where $C_{\rm s},c$ are positive constants independent of $v_i,N_{\rm pde}$ and $R_{\rm c}$. Combine ~\eqref{eq:truncation error} with~\eqref{eq:error-spectral} and taking $C_p = {\rm max}\{\hat{\beta}_i, C_{\rm s}\}>0$ independent of $R_{\rm c}$ and $N_{\rm pde}$, we can obtain the error estimate for~\eqref{eq:pde-error}.
\end{proof}

\bibliographystyle{siamplain}
\bibliography{bib}

\end{document}